\documentclass{amsart}

\usepackage{amssymb, mathtools, delimset}
\usepackage{enumitem, etoolbox}
\usepackage{array, booktabs, floatrow}
\usepackage[nopatch=footnote]{microtype}
\usepackage{xcolor, tikz}
\usetikzlibrary{decorations.pathreplacing}
\tikzset{
graph edge/.style={line width=0.65pt},
graph vertex/.style={circle, fill=black, inner sep=1.3pt},
segment brace/.style={decorate, decoration={brace, mirror, amplitude=3pt}}
}

\usepackage[numbers, sort&compress, nonamebreak, longnamesfirst]{natbib}
\usepackage[
backref=page,
colorlinks=true,
linkcolor=blue,
citecolor=green!40!black,
urlcolor=cyan!45!black
]{hyperref}
\usepackage{aliascnt}
\usepackage[capitalise, nameinlink]{cleveref}

\setlist{itemsep=1pt, parsep=2pt, topsep=0.5pt, partopsep=0pt, leftmargin=17pt}
\allowdisplaybreaks
\numberwithin{equation}{section}
\numberwithin{table}{section}
\numberwithin{figure}{section}
\newcommand{\cone}[1]{\langle#1\rangle}

\newtheorem{theorem}{Theorem}[section]
\newaliascnt{lemma}{theorem}
\newtheorem{lemma}[lemma]{Lemma}
\aliascntresetthe{lemma}
\newaliascnt{proposition}{theorem}
\newtheorem{proposition}[proposition]{Proposition}
\aliascntresetthe{proposition}
\newaliascnt{corollary}{theorem}
\newtheorem{corollary}[corollary]{Corollary}
\aliascntresetthe{corollary}
\newaliascnt{conjecture}{theorem}
\newtheorem{conjecture}[conjecture]{Conjecture}
\aliascntresetthe{conjecture}

\makeatletter
\patchcmd{\NAT@citexnum}
{\NAT@test{\@ne}}
{\NAT@hyper@{\NAT@test{\@ne}}}
{}{\PackageError{csf-caterpillar}{Could not link citation authors}{}}
\let\caterpillar@citexnum\NAT@citexnum
\def\NAT@citexnum[#1][#2]#3{%
\ifNAT@swa\else
\ifnum\NAT@ctype=\z@
\@tempcnta\z@
\@for\@citeb:=#3\do{\advance\@tempcnta\@ne}%
\ifnum\@tempcnta=\tw@
\def\NAT@separator{\ and\NAT@penalty}%
\def\NAT@yrsep{\ and}%
\fi
\fi
\fi
\caterpillar@citexnum[#1][#2]{#3}}
\makeatother

\makeatletter
\appto{\thm@space@setup}{%
\thm@preskip=\parskip
\thm@postskip=0pt\relax}
\newif\ifcaterpillar@inresult
\newdimen\caterpillar@proofsep
\newcommand{\caterpillar@trackresult}[1]{%
\AtBeginEnvironment{#1}{%
\caterpillar@inresulttrue
\global\caterpillar@proofsep=0pt\relax}}
\forcsvlist{\caterpillar@trackresult}
{theorem,lemma,proposition,corollary,conjecture}
\newcommand{\caterpillar@tracklist}{%
\ifcaterpillar@inresult
\ifdim\dimexpr\parskip+4pt\relax>\caterpillar@proofsep
\global\caterpillar@proofsep=\dimexpr\parskip+4pt\relax
\fi
\fi}
\AtBeginEnvironment{enumerate}{\caterpillar@tracklist}
\expandafter\patchcmd\csname\string\proof\endcsname
{\trivlist}{%
\trivlist
\ifdim\caterpillar@proofsep>0pt
\@topsep=\caterpillar@proofsep
\else
\@topsep=3pt
\fi
\@topsepadd=0pt\relax
\global\caterpillar@proofsep=0pt\relax}
{}{\PackageError{csf-caterpillar}{Could not adjust proof spacing}{}}
\makeatother

\makeatletter
\AddToHook{begindocument/end}{%
\def\ltx@label#1{\caterpillar@displaylabel#1\@nil}}
\def\caterpillar@displaylabel{%
\@ifnextchar[\caterpillar@displaylabel@typed\caterpillar@displaylabel@plain}
\def\caterpillar@displaylabel@typed[#1]#2\@nil{\cref@label[#1]{#2}}
\def\caterpillar@displaylabel@plain#1\@nil{\cref@label{#1}}
\makeatother

\crefrangelabelformat{equation}{\textup{(#3#1#4)--(#5#2#6)}}
\crefformat{itm}{\upshape#2#1#3}
\crefformat{wolfgang}{#2Wolfgang's criterion#3}
\Crefformat{wolfgang}{#2Wolfgang's criterion#3}
\creflabelformat{def}{\upshape(#2#1)#3}
\crefname{def}{Def.}{Defs.}
\Crefname{def}{Def.}{Defs.}
\crefname{conjecture}{Conjecture}{Conjectures}
\Crefname{conjecture}{Conjecture}{Conjectures}
\crefname{equation}{Eq.}{Eqs.}
\Crefname{equation}{Eq.}{Eqs.}
\crefname{ineq}{Ineq.}{Ineqs.}
\Crefname{ineq}{Ineq.}{Ineqs.}
\crefname{table}{Tab.}{Tabs.}
\Crefname{table}{Tab.}{Tabs.}
\crefname{figure}{Fig.}{Figs.}
\Crefname{figure}{Fig.}{Figs.}
\crefformat{section}{#2\S#1#3}
\crefrangeformat{section}{#3\S\S#1#4--#5#2#6}
\crefmultiformat{section}{#2\S\S#1#3}{ and~#2#1#3}
{, #2#1#3}{, and~#2#1#3}
\Crefformat{section}{Section~#2#1#3}
\Crefrangeformat{section}{Sections~#3#1#4--#5#2#6}
\Crefmultiformat{section}{Sections~#2#1#3}{ and~#2#1#3}
{, #2#1#3}{, and~#2#1#3}
\crefalias{subsection}{section}

\title{\texorpdfstring{$e$}{e}-Positive caterpillars
have maximum degree at most three}
\author[D.Q.B.~Tang]{Davion Q.B. Tang}
\address[Davion Q.B. Tang]{School of Mathematics and Statistics,
Beijing Institute of Technology,
Beijing 102400,
P. R. China}
\email{davion@bit.edu.cn}

\author[D.G.L.~Wang]{David G.L. Wang\texorpdfstring{\textsuperscript{*}}{}}
\thanks{The work of David G.L. Wang was supported by the General Program of the
National Natural Science Foundation of China (Grant No.~12171034).}
\address[David G.L. Wang]{School of Mathematics and Statistics
\& Beijing Key Laboratory on MCAACI,
Beijing Institute of Technology;
MIIT Key Laboratory of Mathematical Theory and Computation in Information Security,
Beijing 102400,
P. R. China}
\email{glw@bit.edu.cn}

\author[M.M.Y.~Wang]{Monica M.Y. Wang}
\address[Monica M.Y. Wang]{School of Basic Education,
Beijing Institute of Graphic Communication,
Beijing 102600,
P. R. China}
\email{myw@bigc.edu.cn}

\subjclass[2020]{Primary 05E05; Secondary 05C05, 05C15}
\keywords{chromatic symmetric function,
\texorpdfstring{$e$}{e}-positivity,
caterpillar,
specialization,
invariant cone}
\thanks{David G.L. Wang\textsuperscript{*} is the corresponding author.}
\date{}

\makeatletter
\hypersetup{pdftitle={\@title}, pdfauthor={\authors}, pdfkeywords={\@keywords}}
\makeatother

\begin{document}
\bibliographystyle{abbrvnat}
\begin{abstract}
Dahlberg, She, and van Willigenburg conjectured that every $e$-positive tree has maximum degree at most three. We prove this conjecture for caterpillars. Our proof introduces a specialization--cone method: we evaluate chromatic symmetric functions under finite-support specializations that are nonnegative on the elementary symmetric functions, express the evaluations through transfer matrices, and certify negativity using finitely generated cones. Two specializations and three cones handle all caterpillars remaining after a connected-partition reduction, with no bound on the number or spacing of their trivalent vertices. The required cone inclusions and strict sign inequalities are certified by exact rational arithmetic. We also exhibit two infinite families of non-spider caterpillars of maximum degree four that admit connected partitions of every type, showing that Wolfgang's criterion alone leaves cases requiring further obstructions. Our result also reduces Li's criterion on $e$-positivity for caterpillars of maximum degree three to an explicit condition on the positions of trivalent vertices.
\end{abstract}

\maketitle

\section{Introduction}
\citet{Sta95} introduced the chromatic symmetric function
as a generalization of Birkhoff's chromatic polynomial.
For a finite graph~$G$,
its \emph{chromatic symmetric function} is the symmetric function
\[
X_G=\sum_f\prod_{v\in V(G)}x_{f(v)},
\]
where the sum runs over all proper colorings
$f\colon V(G)\to\mathbb Z_{>0}$ of~$G$.
It determines many basic parameters of a graph,
including its numbers of vertices,
edges,
and connected components,
as well as its chromatic number,
independence number,
matching number,
and girth;
for trees,
it also determines the degree sequence,
number of leaves,
and diameter;
see \citet[p.~238, Eq.~(5), Props.~2--3 and Cor.~5]{MMW08}.

The chromatic symmetric function also connects graph coloring
with representation theory and algebraic geometry.
For incomparability graphs of $(3+1)$-free posets,
\citet[Thm.~4]{Gas96} interpreted the Schur coefficients of~$X_G$
as counts of tableaux,
linking these graphs with representations of symmetric groups.
For naturally labeled unit interval graphs,
applying the Hall involution~$\omega$
to the chromatic quasisymmetric function introduced by \citet{SW16}
gives the graded Frobenius characteristic of Tymoczko's dot action
on the cohomology of the corresponding regular semisimple Hessenberg variety,
as proved by \citet[Thm.~129]{BC18}.

For a graph~$G$ on~$n$ vertices
and a basis $\{b_\lambda\colon \lambda\vdash n\}$
of the homogeneous component~$\Lambda_n$ of the ring of symmetric functions,
we call~$G$ \emph{$b$-positive}
if every $b$-coefficient of~$X_G$ is nonnegative.
\citet[p.~186]{Sta95} raised the question
``Which~$X_G$ are $e$-positive?''
Work on this question has established both $e$-positivity
and non-$e$-positivity for various graph classes.

On the positive side,
\citet{Hik24X} proved the \emph{Stanley--Stembridge conjecture}~\citep{SS93},
establishing $e$-positivity for incomparability graphs of $(3+1)$-free posets.
This includes \emph{unit interval graphs} and familiar examples such as
\emph{complete graphs},
\emph{paths},
\emph{twinned paths},
\emph{lariats},
\emph{lollipops},
\emph{melting lollipops},
\emph{twinned lollipops},
\emph{dumbbells},
\emph{generalized bulls},
\emph{$K$-chains},
and \emph{triangular ladders}.
Beyond this setting,
\citet[Cors.~54 and~65]{TV26-vglue} proved $e$-positivity
for vertex-glued sequences of naturally labeled unit interval graphs and cycles,
and for their cyclic closures,
encompassing
\emph{cycles},
\emph{tadpoles},
\emph{hats},
\emph{adjacent cycle-chains} (aka.~\emph{hat chains}),
\emph{kayak paddles},
\emph{clique-path-cycle graphs},
\emph{twinned cycles},
\emph{cycle-chords},
and \emph{sequential cycle-chords}.
Further instances outside the results above include
line graphs of tadpoles,
spiders of the form $S(m+1,\,m,\,1)$,
and \emph{clocks},
see \citet{WW23-JAC},
\citet[p.~2681]{DFv20},
and \citet{CHW26},
respectively.
Here a \emph{spider} $S(a_1,\dots,a_k)$ is the tree graph formed by identifying one endpoint
of each of $k$ paths;
these paths are its \emph{legs},
with positive lengths $a_1,\dots,a_k$.

On the negative side,
\citet{Sta95} observed that neither the \emph{claw}
nor the \emph{net} is $e$-positive;
see \cref{fig:intro-graphs}.
\begin{figure}[ht]
\centering
\begin{tikzpicture}[scale=0.65]
\begin{scope}[xshift=-10cm]
\draw[graph edge] (0,0.9)--(-0.9,0) (0,0.9)--(0.9,0)
(0,0.9)--(0,1.8);
\foreach \p in {(0,0.9),(-0.9,0),(0.9,0),(0,1.8)}
\node[graph vertex] at \p {};
\node at (0,-0.6) {claw};
\end{scope}
\begin{scope}[xshift=-5cm]
\draw[graph edge] (-0.7,0.2)--(0.7,0.2)--(0,1.0)--cycle;
\draw[graph edge] (-0.7,0.2)--(-1.5,0.2)
(0.7,0.2)--(1.5,0.2) (0,1.0)--(0,1.8);
\foreach \p in {(-0.7,0.2),(0.7,0.2),(0,1.0),(-1.5,0.2),(1.5,0.2),(0,1.8)}
\node[graph vertex] at \p {};
\node at (0,-0.6) {net};
\end{scope}
\begin{scope}
\draw[graph edge] (-0.7,1.8)--(0.7,1.8)--(0,0.9)--cycle;
\draw[graph edge] (0,0.9)--(-0.9,0) (0,0.9)--(0.9,0);
\foreach \p in {(-0.7,1.8),(0.7,1.8),(0,0.9),(-0.9,0),(0.9,0)}
\node[graph vertex] at \p {};
\node at (0,-0.6) {cricket};
\end{scope}
\begin{scope}[xshift=5cm]
\draw[graph edge] (-1.8,0)--(-0.9,0)--(0,0)--(0.9,0)--(1.8,0);
\draw[graph edge] (0,0)--(0,0.9)--(0,1.8);
\foreach \p in {(-1.8,0),(-0.9,0),(0,0),(0.9,0),(1.8,0),(0,0.9),(0,1.8)}
\node[graph vertex] at \p {};
\node at (0,-0.6) {$S(2,2,2)$};
\end{scope}
\begin{scope}[xshift=10cm]
\draw[graph edge] (-1.8,0)--(-0.9,0)--(0,0)--(0.9,0)--(1.8,0);
\draw[graph edge] (-0.9,0)--(-0.9,0.9) (0,0)--(0,0.9);
\foreach \p in {(-1.8,0),(-0.9,0),(0,0),(0.9,0),(1.8,0),(-0.9,0.9),(0,0.9)}
\node[graph vertex] at \p {};
\node at (0,-0.6) {$M_2$};
\end{scope}
\end{tikzpicture}
\caption{The claw,
the net,
the cricket,
and the spider~$S(2,2,2)$ are not $e$-positive,
while the non-spider caterpillar~$M_2$ is $e$-positive.}
\label{fig:intro-graphs}
\end{figure}
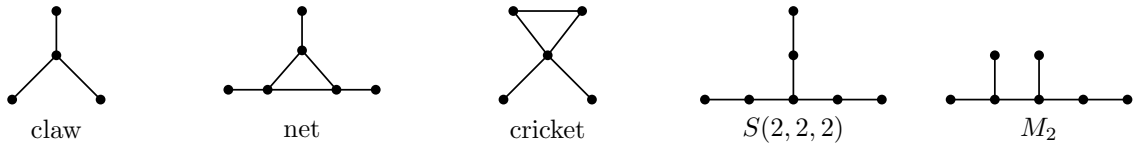
Known non-$e$-positivity results cover a variety of graph families.
Non-Schur-positivity is a stronger obstruction to $e$-positivity,
since every elementary symmetric function is Schur positive.
Results of~\citet{WW20,LLYZ25,LQYZ25,WZ26} cover
connected \emph{unbalanced bipartite graphs},
non-degenerate \emph{squids} other than the \emph{cricket}
(see \cref{fig:intro-graphs}),
incomparability graphs of certain distributive lattices,
and all but finitely many non-degenerate \emph{wheels},
\emph{windmills},
and \emph{fans}.
\citet[Thms.~4.5, 4.10, and~5.5]{DFv20} constructed infinite
non-$e$-positive families that cannot be contracted to a claw,
including a claw-free family.
Further examples include infinite non-$e$-positive families
in each class of $(\mathrm{claw},H)$-free graphs with
$H\in\{\mathrm{diamond},C_4,K_4,2K_2\}$,
\emph{brooms} other than paths and the spider~$S(2,1,1)$,
and \emph{double brooms} that do not degenerate to brooms;
see~\citet{FHM19,WW23-DAM}.

For trees,
\citet{DSv20} proposed the following central conjecture.
\begin{conjecture}[\citeauthor{DSv20}]
\label{conj:e-.tree}
Any $e$-positive tree has maximum degree $\Delta\le 3$.
\end{conjecture}

For a graph~$G$ and a partition $\lambda\vdash \abs{V(G)}$,
a \emph{connected partition} of~$G$ of \emph{type}~$\lambda$
is a partition of~$V(G)$ into \emph{blocks} that induce connected subgraphs
and whose sizes are the parts of~$\lambda$.
\citet[Prop.~1.3.3]{Wol97D} established the following
\emph{connected-partition criterion}.
\begin{proposition}[\citeauthor{Wol97D}]
\label[wolfgang]{prop:Wolfgang}
Let~$G$ be a connected $e$-positive graph on~$n$ vertices.
Then for any integer partition $\lambda\vdash n$,
the graph~$G$ has a connected partition of type~$\lambda$.
\end{proposition}

\citet[\S 1]{BF06} called a connected graph \emph{decomposable}
if it has a connected partition of every type,
and proved that every such tree has $\Delta\le 4$;
see \cref{thm:Barth-Fournier}.
By \cref{prop:Wolfgang},
the same bound holds for $e$-positive trees.
\citet[\S 1]{HW08} studied trees with this property
under the name \emph{arbitrarily vertex decomposable}.
\citet{Li26X} called such a graph \emph{CPET}.
In the chromatic symmetric function literature,
\citet[Thm.~3.12]{Zhe22} proved the bound $\Delta\le 5$,
and \citet[Thm.~1.5]{Tom26} subsequently obtained $\Delta\le 4$
as part of a result on cut vertices.
\Citet[Thm.~1.6]{Tom26} proved the conjectured bound for spiders by excluding all
$4$-legged spiders.
\citet[Thms.~2.5 and~3.4--3.5; \S 4]{Li26X}
derived quantitative conditions for $e$-positivity of trees
using connected-partition counts and acyclic orientations,
including a necessary inequality for \emph{caterpillars}
obtained by counting orientations with two sinks.
We prove the conjecture uniformly for caterpillars.

\begin{theorem}\label{thm:e-.caterpillar}
Any $e$-positive caterpillar~$T$ has maximum degree $\Delta(T)\le 3$.
\end{theorem}

The bound~$3$ is sharp:
\citet[Eg.~6.5]{Zhe22} recorded the $e$-positive non-spider
caterpillar~$M_2$ of $\Delta=3$ shown in \cref{fig:intro-graphs}.
\citet[\S 3]{CHS07} showed that caterpillars are $S(2,2,2)$-free trees;
see \cref{fig:intro-graphs}.
The classes of caterpillars and spiders overlap,
but neither contains the other.
Caterpillars also feature in the \emph{tree isomorphism problem}:
\citet[Prop.~10]{MMW08} showed that
the chromatic symmetric function recognizes caterpillars among trees,
and \citet[Thm.~2]{LS19} proved that it distinguishes individual caterpillars.
See also
\citet[Thm.~1.1]{GNRS24} for a more general setting.

\citet{Tom26} constructed an infinite family of CPET $4$-legged spiders
that are not $e$-positive,
showing that the criterion alone cannot settle the spider case.
In \cref{cor:infinite.CPET},
we give two infinite families of non-spider CPET caterpillars
of maximum degree~$4$,
one arising from a construction of \citet{CGMPW06}.
These families show that the CPET condition alone is insufficient
to establish \cref{thm:e-.caterpillar}.

Existing approaches to non-$e$-positivity fall into two broad classes:
structural obstructions and direct detection of negative $e$-coefficients.
The first approach includes \cref{prop:Wolfgang},
the stable-partition criterion of \citet{Sta98},
the strongly-niceness criterion of \citet{LLYZ25},
and the combinatorial formula of \citet{WW20}
for every Schur coefficient.
A second approach seeks a specific negative $e$-coefficient.
For example,
\citet{Tom26} ruled out $e$-positivity for $4$-legged spiders
by computing such a coefficient.
Further applications appeared in \citet{WW23-DAM}.

We introduce the \emph{specialization--cone method}
as our second main contribution.
We choose a finite-support specialization~$\theta$ satisfying
$\theta(e_j)\ge 0$ for every $j\ge 1$.
The argument combines three mathematical ingredients.
First,
the transfer formula in \cref{lem:transfer}
gives a finite-dimensional representation of~$\theta(X_T)$
as a bilinear form in the two side evaluation rows,
whose dimension depends only on~$\theta$.
Second,
invariant finitely generated cones contain all admissible side evaluation rows.
Third,
strict negativity of this form on the relevant pairs of cone generators
extends to their nonzero nonnegative combinations,
giving $\theta(X_T)<0$ and hence ruling out $e$-positivity.
Cone invariance controls arbitrarily long products
of the relevant transfer matrices.
Consequently,
these finite certificates handle an unbounded number of trivalent vertices
and arbitrary spacings permitted by the structural reduction.
In establishing \cref{thm:e-.caterpillar},
we implement this method using only two finite-support specializations
$\varphi$ and~$\psi$ and three row cones.

This is the first method
to combine specializations and invariant cones,
to the best of our knowledge,
for detecting non-$e$-positivity.
Both ingredients have precedents.
\citet[\S 4]{HJLOY21} specialized the elementary generators by $e_n\mapsto t$,
giving a new derivation of the sink theorem of \citet{Sta95}.
\citet[\S\S 3 and~4]{CC24} introduced the tree polynomial
through the specialization $X_{P_n}\mapsto x$,
interpreting the coefficients in terms of certain spanning trees
and edge cutsets.
Invariant cones are classical;
see, for example,
\citet{Bir67IC}.
Here we certify the required cone inclusions
by nonnegative rational factorizations.
Matrix methods have also appeared in the study of chromatic symmetric functions:
\citet{TV26-vglue} used matrix multiplication to encode vertex gluing.

This paper is organized as follows.
\Cref{sec:framework} describes untreated caterpillars
and expresses their chromatic symmetric functions in matrix products.
It establishes the framework of the specialization--cone method.
\Cref{sec:nonadj,sec:adj}
handle the cases without and with a trivalent neighbor of the degree-$4$ center,
respectively.
\Cref{sec:CPET} proves the additional CPET results.
\Cref{sec:Li} gives two explicit examples
and a necessary condition for $e$-positivity of caterpillars with $\Delta=3$.

\section{The structural reduction and matrix transfer}
\label{sec:framework}

This section is devoted to \cref{prop:T.abcd,lem:transfer}.
The former parameterizes untreated caterpillars,
and the latter expresses their chromatic symmetric functions in matrix products.

Write $[n]=\{1,\dots,n\}$ for $n\ge 0$.
We use the standard word terminology in \citet[Ch.~1]{Lot97B}.
A \emph{word} is a finite sequence of \emph{letters} from an alphabet.
The empty word is denoted by~$\epsilon$.
For a word $w=w_1\cdots w_l$,
its \emph{reversal} is $\overline w=w_l\cdots w_1$.
The \emph{concatenation} of words~$u$ and~$v$ is denoted by~$uv$.
For $k\ge 0$,
the power~$w^k$ denotes the concatenation of~$k$ copies of~$w$,
with $w^0=\epsilon$.
A \emph{factor} of~$w$ is a word~$z$ such that
$w$ has a \emph{factorization} $w=pzq$ for some words~$p$ and~$q$.
It is a \emph{prefix} if $p=\epsilon$,
and a \emph{suffix} if $q=\epsilon$.
A \emph{composition} is a word $I=i_1\cdots i_\ell$
over~$\mathbb Z_{>0}$,
with \emph{size} $\abs{I}=i_1+\cdots+i_\ell$.
We also write $I$ as the list of parts $i_j$.
A \emph{partition} $\lambda\vdash n$ is a multiset of positive integers
with sum~$n$,
identified with the composition whose parts are weakly decreasing.
The empty word is the unique composition of~$0$.

For general CPET trees,
\citet[Cor.~1]{BF06} proved the following bound.

\begin{theorem}[\citeauthor{BF06}]
\label{thm:Barth-Fournier}
Every CPET tree has maximum degree at most~$4$.
\end{theorem}

\Citet[Thm.~21 and Cors.~22--23]{HW08} proved further structural restrictions
for CPET caterpillars.

\begin{theorem}[\citeauthor{HW08}]
\label{thm:Hornak-Wozniak}
Any CPET tree~$G$ contains at most one path of odd order at least~$3$
whose endpoints are leaves of~$G$
and whose vertices other than the endpoints and their neighbors
have degree~$2$ in~$G$.
Consequently,
a CPET caterpillar has at most one vertex of degree~$4$;
any such vertex has exactly two leaf neighbors.
\end{theorem}

We can now describe the structure of CPET caterpillars with~$\Delta\ge 4$.

\begin{proposition}
\label{prop:T.abcd}
Every CPET caterpillar with $\Delta\ge 4$
must be of the form $T(a;\mathbf b\,\vert\,\mathbf c;d)$,
defined as follows:
it is obtained from two paths
meeting only at their common endpoint~$v_\ast$
by attaching two leaves there
and one leaf at each marked internal vertex
$u_1$,
$\dots$,
and~$u_s$ on one path and~$v_1$,
$\dots$,
and~$v_t$ on the other.
These vertices are indexed outward from~$v_\ast$.
The successive segment lengths are~$b_1$,
$\dots$,
$b_s$,
and~$a$ on the $u$-side
and~$c_1$,
$\dots$,
$c_t$,
and~$d$ on the $v$-side,
where~$a$ and~$d$ are positive even integers
and $\mathbf b=b_1\dotsm b_s$ and $\mathbf c=c_1\dotsm c_t$
are possibly empty compositions with odd parts.
\end{proposition}

\begin{proof}
Let~$G$ be a CPET caterpillar with $\Delta\ge 4$.
By \cref{thm:Barth-Fournier,thm:Hornak-Wozniak},
$G$ has a unique vertex~$v_\ast$ of degree~$4$ with exactly two leaf neighbors,
say, $x$ and~$y$.
The path~$xv_\ast y$ has order~$3$.
By the uniqueness of such a path in \cref{thm:Hornak-Wozniak},
it is elementary to find the desired parities of all parameters.
\end{proof}

The same construction defines $T(a;\mathbf b\,\vert\,\mathbf c;d)$
for arbitrary positive terminal lengths and possibly empty compositions.
We call~$v_\ast$ the \emph{center}.
The segments of lengths~$b_i$ and~$c_j$ are the \emph{gaps},
and those of lengths~$a$ and~$d$ are the \emph{terminals};
see \cref{fig:T.abcd}.
For positive terminal lengths,
the two components left after removing~$v_\ast$ and its two attached leaves
are the \emph{sides}.
Their \emph{weights},
namely their numbers of vertices,
are $a+\sum_{i=1}^s(b_i+1)$ and $d+\sum_{j=1}^t(c_j+1)$, respectively.

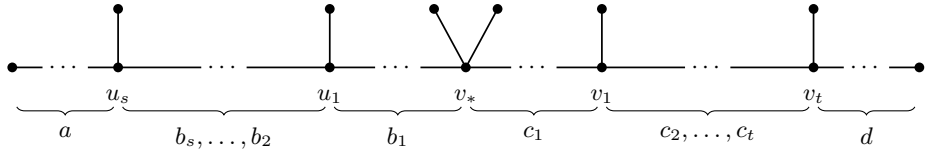
\begin{figure}[ht]
\centering
\begin{tikzpicture}[
x=1.0cm, y=0.84cm,
every node/.style={font=\small}]
\coordinate (L) at (-6,0);
\coordinate (ur) at (-4.6,0);
\coordinate (uone) at (-1.8,0);
\coordinate (center) at (0,0);
\coordinate (vone) at (1.8,0);
\coordinate (vs) at (4.6,0);
\coordinate (R) at (6,0);
\coordinate (adotswest) at (-5.6,0);
\coordinate (adotseast) at (-5.0,0);
\coordinate (leftdotswest) at (-3.5,0);
\coordinate (leftdotseast) at (-2.9,0);
\coordinate (bunodotswest) at (-1.2,0);
\coordinate (bunodotseast) at (-0.6,0);
\coordinate (cunodotswest) at (0.6,0);
\coordinate (cunodotseast) at (1.2,0);
\coordinate (rightdotswest) at (2.9,0);
\coordinate (rightdotseast) at (3.5,0);
\coordinate (ddotswest) at (5.0,0);
\coordinate (ddotseast) at (5.6,0);
\draw[graph edge] (L)--(adotswest);
\draw[graph edge] (adotseast)--(ur)--(leftdotswest);
\draw[graph edge] (leftdotseast)--(uone)--(bunodotswest);
\draw[graph edge] (bunodotseast)--(center)--(cunodotswest);
\draw[graph edge] (cunodotseast)--(vone)--(rightdotswest);
\draw[graph edge] (rightdotseast)--(vs)--(ddotswest);
\draw[graph edge] (ddotseast)--(R);
\draw[graph edge] (ur)--++(0,0.92);
\draw[graph edge] (uone)--++(0,0.92);
\draw[graph edge] (center)--++(-0.42,0.92);
\draw[graph edge] (center)--++(0.42,0.92);
\draw[graph edge] (vone)--++(0,0.92);
\draw[graph edge] (vs)--++(0,0.92);
\foreach \p in {L,ur,uone,center,vone,vs,R}
\node[graph vertex] at (\p) {};
\foreach \p in {(-4.6,0.92),(-1.8,0.92),(-0.42,0.92),
(0.42,0.92),(1.8,0.92),(4.6,0.92)}
\node[graph vertex] at \p {};
\node[inner sep=1.2pt] at (-5.3,0) {$\cdots$};
\node[inner sep=1.2pt] at (-3.2,0) {$\cdots$};
\node[inner sep=1.2pt] at (-0.9,0) {$\cdots$};
\node[inner sep=1.2pt] at (0.9,0) {$\cdots$};
\node[inner sep=1.2pt] at (3.2,0) {$\cdots$};
\node[inner sep=1.2pt] at (5.3,0) {$\cdots$};
\node[below=5pt] at (ur) {$u_s$};
\node[below=5pt] at (uone) {$u_1$};
\node[below=5pt] at (center) {$v_\ast$};
\node[below=5pt] at (vone) {$v_1$};
\node[below=5pt] at (vs) {$v_t$};
\draw[segment brace]
(-5.94,-0.62)--node[below=4pt] {$a$} (-4.66,-0.62);
\draw[segment brace]
(-4.54,-0.62)--node[below=4pt] {$b_s,\dots,b_2$} (-1.86,-0.62);
\draw[segment brace]
(-1.74,-0.62)--node[below=4pt] {$b_1$} (-0.06,-0.62);
\draw[segment brace]
(0.06,-0.62)--node[below=4pt] {$c_1$} (1.74,-0.62);
\draw[segment brace]
(1.86,-0.62)--node[below=4pt] {$c_2,\dots,c_t$} (4.54,-0.62);
\draw[segment brace]
(4.66,-0.62)--node[below=4pt] {$d$} (5.94,-0.62);
\end{tikzpicture}
\caption{The caterpillar $T(a;\mathbf b\,\vert\,\mathbf c;d)$
with center~$v_\ast$,
gap lengths~$b_i$ and~$c_j$,
and terminal lengths~$a$ and~$d$.}
\label{fig:T.abcd}
\end{figure}

Let $\Lambda$ be the ring of symmetric functions.
Fix an integer~$m\ge 2$ and a $\mathbb Q$-algebra specialization
$\theta\colon\Lambda\to\mathbb Q$
such that $\theta(e_m)\ne 0$ and $\theta(e_j)=0$ for $j\ge 1$ unless $2\le j\le m$.
Write $L_n=\theta(X_{P_n})$ for $n\ge 1$ and set $L_0=1$.

\begin{lemma}
\label{lem:rec.L}
$L_n=\theta(e_n)+\sum_{j=2}^n(j-1)\theta(e_j)L_{n-j}$ for $n\ge 1$.
\end{lemma}

\begin{proof}
Specializing the path generating function of
\citet[Prop.~5.3]{Sta95} gives
\[
\sum_{n\ge 0}L_n z^n
=\frac{1+\sum_{j=2}^m\theta(e_j)z^j}
{1-\sum_{j=2}^m(j-1)\theta(e_j)z^j}.
\]
Comparing coefficients of~$z^n$ proves the formula.
\end{proof}

Let~$\mathbf e_i$ be the~$i$th standard column of~$\mathbb Q^m$,
and let the superscript~$\top$ denote transpose.
Then the outer product $\mathbf e_i\mathbf e_j^\top$
has a single nonzero entry~$1$
in row~$i$ and column~$j$.
Define the $m\times m$ matrix
\begingroup
\renewcommand{\arraystretch}{1.4}
\begin{equation}
\label[def]{def:A}
A=\brk4{
\begin{array}{c|cccc}
\mathbf 0&\multicolumn{4}{c}{I_{m-1}}
\\
\hline
(m-1)\theta(e_m)&(m-2)\theta(e_{m-1})&\dotsm&\theta(e_2)&0
\end{array}
},
\end{equation}
\endgroup
where~$I_{m-1}$ is the identity matrix of order~$m-1$.
The last row encodes the recurrence in \cref{lem:rec.L}.
Thus~$A$ is the \emph{companion matrix} of the characteristic polynomial
\[
\det(zI_m-A)
=z^m-\sum_{j=2}^m(j-1)\theta(e_j)z^{m-j}.
\]
In particular,
iterating
$\mathbf e_i^\top A=\mathbf e_{i+1}^\top$
from~$i=1$ gives
\begin{equation}
\label{e1:A.i-1=ei}
\mathbf e_1^\top A^{i-1}=\mathbf e_i^\top
\qquad\text{for $1\le i\le m$.}
\end{equation}
Put $\ell=(L_2,\,\dots,\,L_{m+1})^\top$ and define
\begin{equation}
\label[def]{def:M}
M_j=(A^2-\ell\mathbf e_1^\top)A^{j-1},
\qquad
M_{(b_1,\dots,b_s)}=M_{b_s}\cdots M_{b_1},
\qquad\text{and}\qquad
M_\epsilon=I_m.
\end{equation}
Define the \emph{side evaluation row}
\begin{equation}
\label[def]{def:r.ab}
r_{a,\mathbf b}=\mathbf e_1^\top A^{a-1}M_{\mathbf b}
\qquad\text{for $a\ge 1$.}
\end{equation}
Finally,
define the \emph{initial spider matrix}~$W_\theta$ by
\begin{equation}
\label[def]{def:W.initial}
(W_\theta)_{ij}=\theta(X_{S(i,j,1,1)})
\qquad\text{for $1\le i,j\le m$.}
\end{equation}
It is symmetric.
Note that $S(i,j,1,1)=T(i;\epsilon\,\vert\,\epsilon;j)$.

\begin{lemma}[Transfer invertibility]
\label{lem:inv}
The matrices~$A$ and~$M_j$ are invertible for every $j\ge 1$.
As a consequence,
$r_{a,\mathbf b}\ne 0$ for every positive integer~$a$
and possibly empty composition~$\mathbf b$.
\end{lemma}

\begin{proof}
Expansion along the first column of $A$ gives
\[
\det A=(-1)^{m-1}(m-1)\theta(e_m)\ne 0.
\]
Put $v=A^{-2}\ell$.
The path recurrence gives $A v=(L_1,\dots,L_m)^\top$,
so $v_{i+1}=L_i$ for $i<m$.
The last row of this identity and \cref{lem:rec.L} at $n=m$ give
\[
(m-1)\theta(e_m)v_1
=L_m-\sum_{j=2}^{m-1}(j-1)\theta(e_j)L_{m-j}
=m\theta(e_m).
\]
Note that $M_1=A^2-\ell\mathbf e_1^\top
=A^2(I_m-v\mathbf e_1^\top)$.
The second factor differs from~$I_m$ only in its first column.
Expansion along this column gives its determinant as $1-v_1$.
Thus
\[
\det M_1
=\det A^2(1-v_1)
=-(m-1)\theta(e_m)^2\ne 0.
\]
Since $M_j=M_1A^{j-1}$,
all gap transfers are invertible.
By \cref{def:r.ab}, $r_{a,\mathbf b}\ne 0$.
\end{proof}

\citet[p.~1780]{KS25} call a star subgraph a \emph{leaf star}
if each of its edges is incident with a leaf of the ambient graph.
Let $T$ be a caterpillar.
A \emph{spine} of $T$ is a path containing all its nonleaf vertices
and having two leaves of $T$ as endpoints.
Given a chosen spine~$P$ of~$T$,
the \emph{internal leaf star} at a vertex~$v\in V(P)$
is the subgraph of~$T$ induced by~$v$
and all its leaf neighbors outside~$P$.
The vertex sets of these stars partition~$V(T)$.
For $T=T(a;\mathbf b\,\vert\,\mathbf c;d)$ with $a,d\ge 1$,
we choose~$P$ to be the union of the two defining paths.
The internal leaf star at $v_\ast$ has order~$3$,
and every other internal leaf star has order~$2$ or~$1$.

\begin{lemma}[Transfer formula]
\label{lem:transfer}
We have
$\theta\brk1{X_{T(a;\mathbf b\,\vert\,\mathbf c;d)}}
=r_{a,\mathbf b}W_\theta r_{d,\mathbf c}^\top$
for positive integers~$a$ and~$d$
and possibly empty compositions~$\mathbf b$ and~$\mathbf c$.
\end{lemma}

\begin{proof}
Let $T=T(a;\mathbf b\,\vert\,\mathbf c;d)$.
By Vieta's formulas,
there exist~$x_1,\dots,x_m\in\mathbb C$ such that
\[
e_j(x_1,\dots,x_m)=\theta(e_j)
\qquad\text{for $1\le j\le m$}.
\]
In particular, $\sum_i x_i=\theta(e_1)=0$.
Put
\[
\mathbf x=
\begin{pmatrix}
x_1
\\
x_2
\\
\vdots
\\
x_m
\end{pmatrix},
\qquad
D=
\begin{pmatrix}
x_1&&&
\\
&x_2&\multicolumn{2}{c}{\raisebox{.6ex}[0pt][0pt]{\huge$0$}}
\\
\multicolumn{2}{c}{\raisebox{-.6ex}[0pt][0pt]{\huge$0$}}&\ddots&
\\
&&&x_m
\end{pmatrix},
\qquad\text{and}\qquad
K=
\begin{pmatrix}
0&&&
\\
&0&\multicolumn{2}{c}{\raisebox{.6ex}[0pt][0pt]{\huge$1$}}
\\
\multicolumn{2}{c}{\raisebox{-.6ex}[0pt][0pt]{\huge$1$}}&\ddots&
\\
&&&0
\end{pmatrix},
\]
where~$K$ has order~$m$.
Let $P=KD$ and $Q=-KD^2$.
Then
\begingroup
\renewcommand{\arraystretch}{1.4}
\[
P=
\begin{pmatrix}
0&x_2&x_3&\dots&x_m
\\
x_1&0&x_3&\dots&x_m
\\
x_1&x_2&0&\dots&x_m
\\
\vdots&\vdots&\vdots&\ddots&\vdots
\\
x_1&x_2&x_3&\dots&0
\end{pmatrix}
\qquad\text{and}\qquad
Q=
\begin{pmatrix}
0&-x_2^2&-x_3^2&\dots&-x_m^2
\\
-x_1^2&0&-x_3^2&\dots&-x_m^2
\\
-x_1^2&-x_2^2&0&\dots&-x_m^2
\\
\vdots&\vdots&\vdots&\ddots&\vdots
\\
-x_1^2&-x_2^2&-x_3^2&\dots&0
\end{pmatrix}.
\]
\endgroup
Use the chosen spine~$P$ of~$T$.
Build either side from its terminal toward the center,
one internal leaf star at a time.
For the subtree~$H$ of a side constructed so far,
let~$v$ be the last spine vertex reached.
We introduce the \emph{color row} of~$H$,
defined by $\rho(H)=(\rho_1,\dots,\rho_m)$,
where
\[
\rho_i=\sum_{\substack{\kappa\colon V(H)\to[m]\text{ proper}
\\
\kappa(v)=i}}
\prod_{u\in V(H)}x_{\kappa(u)}.
\]
Write the all-one column of length~$m$ as~$\mathbf 1$.
Then
\[
\theta(X_H)
=X_H(x_1,\dots,x_m)
=\rho_1+\dots+\rho_m
=\rho(H)\mathbf 1.
\]

Appending a spine vertex alone transfers~$\rho=\rho(H)$ to~$\rho P$,
since the updated~$j$th coordinate is
\[
x_j\sum_{i\ne j}\rho_i
=\sum_{i=1}^m \rho_i P_{ij}
=(\rho P)_j.
\]
In particular,
\begin{equation}
\label{pf:L=xP.n-1}
L_n
=\theta(X_{P_n})
=\mathbf x^\top P^{n-1}\mathbf 1
\qquad\text{for $n\ge 1$}.
\end{equation}
On the other hand,
appending an internal leaf star of order~$2$ transfers~$\rho$ to~$\rho Q$,
since this star contributes the weight
$x_i\sum_{j\ne i}x_j=-x_i^2$ when the spine vertex has color~$i$.

We now relate these transfers to the matrices~$A$ and~$M_1$.
The~$i$th column of $zI_m-P$ is $(z+x_i)\mathbf e_i-x_i\mathbf 1$.
Expand the determinant by multilinearity in the columns.
Any term choosing~$\mathbf 1$ in at least two columns vanishes,
since those columns are proportional.
Thus
\[
\det(zI_m-P)
=\prod_{i=1}^m(z+x_i)
-\sum_{i=1}^m x_i\prod_{k\ne i}(z+x_k).
\]
For each set~$I\subseteq[m]$ of size~$j$,
the product $z^{m-j}\prod_{i\in I}x_i$
occurs once in the first product and~$j$ times in the sum.
Summing over~$I$ gives
\begin{equation}
\label{CharPoly.P}
\det(zI_m-P)
=z^m-\sum_{j=2}^m(j-1)\theta(e_j)z^{m-j}.
\end{equation}
By the \emph{Cayley--Hamilton theorem}
(see \citet[Thm.~2.4.3.2]{HJ13B}), \cref{CharPoly.P} gives
\begin{equation}
\label{Cayley-Hamilton:P}
P^m=\sum_{j=2}^m(j-1)\theta(e_j)P^{m-j}.
\end{equation}
Define the $m\times m$ matrix~$S$ to have~$i$th row $\mathbf x^\top P^{i-1}$.
We claim that
\begin{equation}
\label{SP=AS}
SP=AS.
\end{equation}
In fact,
for $1\le i<m$,
the $i$th row of~$AS$ is the $(i+1)$st row of~$S$,
namely $\mathbf x^\top P^i$,
which is also the $i$th row of~$SP$.
By \cref{def:A,Cayley-Hamilton:P},
the last row of~$AS$ is
\[
\sum_{j=2}^m(j-1)\theta(e_j)\mathbf x^\top P^{m-j}
=\mathbf x^\top P^m,
\]
which is the last row of~$SP$.
This proves the claimed \cref{SP=AS}.

Since $P\mathbf 1=-\mathbf x$,
\cref{pf:L=xP.n-1} gives
\[
(S \mathbf x)_i
=\mathbf x^\top P^{i-1}\mathbf x
=-\mathbf x^\top P^i\mathbf 1
=-L_{i+1}
\qquad\text{for $1\le i\le m$}.
\]
Thus $S \mathbf x=-\ell$.
We also have $Q=P^2+\mathbf x \mathbf x^\top$
and $\mathbf x^\top=\mathbf e_1^\top S$.
As a consequence,
\[
S Q
=SP^2+S \mathbf x \mathbf x^\top
=A^2 S-\ell \mathbf e_1^\top S
=M_1 S.
\]
Applying \cref{SP=AS} repeatedly, we find
\[
SQP^{t-1}
=M_1SP^{t-1}
=M_1A^{t-1}S
=M_tS
\qquad\text{for $t\ge 1$.}
\]
Starting with the terminal and appending the internal leaf stars,
we obtain the color row of the $\mathbf b$-side:
\begin{align*}
\mathbf x^\top P^{a-1}(QP^{b_s-1})\cdots(QP^{b_1-1})
&=\mathbf e_1^\top A^{a-1}S(QP^{b_s-1})\cdots(QP^{b_1-1})
\\
&=\mathbf e_1^\top A^{a-1}M_{b_s}\cdots M_{b_1}S
=r_{a,\mathbf b}S.
\end{align*}
The other side similarly has color row~$r_{d,\mathbf c}S$.
Fix the color~$i$ of~$v_\ast$.
Then the weight of the internal leaf star at~$v_\ast$ is $x_i(\sum_{j\ne i}x_j)^2=x_i^3$.
Therefore,
the contribution of the $\mathbf b$-side to~$X_T$ is
\[
\sum_j(r_{a,\mathbf b}S)_jK_{ji}
=(r_{a,\mathbf b}SK)_i.
\]
Similarly,
the second side contributes $(KS^\top r_{d,\mathbf c}^\top)_i$.
Summing over the color~$i$ at~$v_\ast$ gives
\begin{equation}
\label{pf:XT:rWr}
\theta(X_T)
=\sum_{i=1}^m(r_{a,\mathbf b}SK)_i\,x_i^3\,
(KS^\top r_{d,\mathbf c}^\top)_i
=r_{a,\mathbf b}M r_{d,\mathbf c}^\top,
\qquad\text{where $M=S K D^3 K S^\top$.}
\end{equation}
Let $1\le i,j\le m$.
By \labelcref{e1:A.i-1=ei,pf:XT:rWr},
\[
M_{ij}
=\mathbf e_i^\top M\mathbf e_j
=r_{i,\epsilon} M r_{j,\epsilon}^\top
=\theta(X_{T(i;\epsilon\,\vert\,\epsilon;j)})
=\theta(X_{S(i,j,1,1)})
=(W_\theta)_{ij}.
\]
Hence $M=W_\theta$.
\end{proof}

\section{A one-cone proof of the nonadjacent case}
\label{sec:nonadj}

In this section,
we treat caterpillars whose center has no trivalent neighbor,
see \cref{prop:nonadj}.
We apply the framework of \cref{sec:framework}
with $m=4$ and $\theta=\varphi$,
where
\[
\varphi(e_2)=1,
\qquad
\varphi(e_3)=1/2,
\qquad\text{and}\qquad
\varphi(e_4)=1/4.
\]
By \cref{def:A,def:M,def:W.initial},
\begingroup
\setlength{\arraycolsep}{4pt}
\[
A=
\begin{pmatrix}
0&1&0&0
\\
0&0&1&0
\\
0&0&0&1
\\
3/4&1&1&0
\end{pmatrix},
\quad
M_1=
\begin{pmatrix}
-2&0&1&0
\\
-3/2&0&0&1
\\
-9/4&1&1&0
\\
-7/2&3/4&1&1
\end{pmatrix},
\quad\text{and}\quad
W_\varphi=-\frac1{32}
\begin{pmatrix}
80&72&124&176
\\
72&44&80&104
\\
124&80&152&204
\\
176&104&204&265
\end{pmatrix}.
\]
\endgroup
Direct multiplication gives
\begin{equation}
\label{eq:fourth-order-transfer-polynomial}
M_1^4
=\frac32M_1^2+\frac12M_1+\frac3{16}I_4.
\end{equation}
The \emph{row cone} generated by a matrix~$F$,
denoted by~$\cone{F}$,
is the set of all nonnegative linear combinations of its rows.
Let $D_j=M_1^j A^2$ for $0\le j\le 3$, and
\[
F_\varphi=
\begin{pmatrix}
2&3&-3&1
\\
1&1&2&-1
\\
-1&-4&-1&4
\\
-3&4&3&-1
\end{pmatrix}.
\]
The following certificate gives a row cone invariant under these four matrices.
All matrix and vector inequalities in this paper are entrywise.

\begin{lemma}
\label{lem:cone.nonadj}
We have $\mathbf e_2^\top\in\cone{F_\varphi}$,
$\cone{F_\varphi}D_j\subseteq\cone{F_\varphi}$ for each $j\in\{0,1,2,3\}$,
and
$F_\varphi W_\varphi F_\varphi^\top<0$.
\end{lemma}

\begin{proof}
Direct calculation gives
$\mathbf e_2^\top
=(31,\,11,\,1,\,24)F_\varphi/196\in\cone{F_\varphi}$
and $F_\varphi D_j=D_j'F_\varphi$,
where the matrices~$D_j'$ are all strictly positive
and defined respectively by
\begin{align*}
D_0'&=\frac1{784}
\begin{pmatrix}
76&8&856&356
\\
209&1149&247&48
\\
113&489&187&372
\\
912&1076&668&156
\end{pmatrix},
&
D_1'&=\frac1{784}
\begin{pmatrix}
655&87&293&368
\\
634&10&90&4
\\
709&397&339&100
\\
851&479&1077&564
\end{pmatrix},
\\[6pt]
D_2'&=\frac1{3136}
\begin{pmatrix}
2246&3130&4454&2036
\\
165&1753&3331&1544
\\
1575&1855&3661&1820
\\
4716&3096&5520&3316
\end{pmatrix},
&\text{and}\quad
D_3'&=\frac1{6272}
\begin{pmatrix}
10033&3617&7171&4916
\\
6995&579&1585&2172
\\
7952&3024&5656&4032
\\
13671&9359&15533&8820
\end{pmatrix}.
\end{align*}
Thus every row of~$F_\varphi D_j$
belongs to~$\cone{F_\varphi}$.
Finally,
\[
F_\varphi W_\varphi F_\varphi^\top
=-\frac1{32}
\begin{pmatrix}
389&475&292&411
\\
475&581&476&505
\\
292&476&272&112
\\
411&505&112&17
\end{pmatrix}<0.\qedhere
\]
\end{proof}

\begin{proposition}
\label{prop:nonadj}
\Cref{thm:e-.caterpillar} holds when
the center of the caterpillar has no trivalent neighbor.
\end{proposition}

\begin{proof}
By \cref{prop:Wolfgang},
it suffices to treat the caterpillars described in \cref{prop:T.abcd}.
Write $T=T(a;\mathbf b\,\vert\,\mathbf c;d)$
with $\mathbf b=(b_1,\dots,b_s)$ having odd parts~$b_i$
and with even~$a\ge 2$.
Write $b_i=2k_i+1$ with $k_i\in\mathbb Z_{\ge 0}$.
Since $v_\ast$ has no trivalent neighbor,
we have $k_1\ge 1$ when $s\ge 1$.
By \labelcref{e1:A.i-1=ei},
$\mathbf e_1^\top A=\mathbf e_2^\top$.
By \cref{def:r.ab},
\begin{equation}
\label{pf:r.ab}
r_{a,\mathbf b}
=\mathbf e_2^\top A^{a-2}
(M_1 A^{2k_s})\cdots(M_1 A^{2k_1}).
\end{equation}
Thus $r_{a,\mathbf b}=\mathbf e_2^\top V$ for a word~$V$
in~$A^2$ and~$M_1$ that is either empty or ends with~$A^2$.
Splitting~$V$ after each occurrence of~$A^2$ gives
\[
V=(M_1^{u_1}A^2)\cdots(M_1^{u_\ell}A^2),
\qquad\text{where $\ell\ge 0$ and $u_i\ge 0$ for $1\le i\le \ell$.}
\]
By \cref{eq:fourth-order-transfer-polynomial},
each factor is a nonnegative linear combination of~$D_0=A^2$,
$D_1$,
$D_2$,
and~$D_3$.
Thus
$r_{a,\mathbf b}\in\cone{F_\varphi}$
by \cref{lem:cone.nonadj}.
The same argument applies to the other side.
We may therefore write
$r_{a,\mathbf b}=\lambda F_\varphi$
and
$r_{d,\mathbf c}=\mu F_\varphi$
for nonnegative rows~$\lambda$ and~$\mu$.
By \cref{pf:r.ab,lem:inv},
the side evaluation rows $r_{a,\mathbf b}$ and $r_{d,\mathbf c}$ are nonzero.
Thus~$\lambda$ and~$\mu$ are nonzero.
Hence \cref{lem:transfer,lem:cone.nonadj} give
\[
\varphi(X_T)
=\lambda(F_\varphi W_\varphi F_\varphi^\top)\mu^\top<0,
\]
which rules out $e$-positivity.
\end{proof}

\section{A two-cone proof of the adjacent case}
\label{sec:adj}

In this section,
we treat the remaining case of \cref{thm:e-.caterpillar}.

\begin{proposition}
\label{prop:adj}
\Cref{thm:e-.caterpillar} holds when
the center of the caterpillar has a trivalent neighbor.
\end{proposition}

We will prove \cref{prop:adj} in three steps.
First, 
we encode each remaining caterpillar by a word
and use a connected partition to obtain a side consisting of blocks of size~$3$.
Second, 
we prove the boundary restrictions needed to choose initial prefixes
of the side factors.
Finally, 
we use two row cones to obtain the required negative evaluations.

By \cref{prop:Wolfgang},
it suffices to consider a CPET caterpillar $T=T(a;\mathbf b\mid \mathbf c;d)$
as in \cref{prop:T.abcd}.
Since the center has a trivalent neighbor,
we may suppose that $b_1=1$ by symmetry.
The order of~$T$ is the odd number
\begin{equation}
\label{N.abcd}
N=a+d+\sum_i(b_i+1)+\sum_j(c_j+1)+3.
\end{equation}
Let~$P$ be the chosen spine of~$T$.
Unless otherwise stated,
we read along~$P$ from the endpoint on the $\mathbf b$-side
toward the endpoint on the $\mathbf c$-side.
List its vertices in this order as~$z_1$,
$\dots$,
and~$z_k$.
For $1\le i\le k$,
let~$G_i$ be the internal leaf star at~$z_i$ relative to~$P$.
Define
\[
w(T)=w_1w_2\cdots w_k,\quad\text{where $w_i=\abs{V(G_i)}$.}
\]
We define the two \emph{side factors}~$\beta$ and~$\gamma$ by
$w(T)=\beta3\overline{\gamma}$,
where
\begin{equation}
\label[def]{beta:gamma}
\beta=1^a2 1^{b_s-1}\cdots2 1^{b_1-1}
\qquad\text{and}\qquad
\overline{\gamma}=1^{c_1-1}2\cdots1^{c_t-1}2 1^d.
\end{equation}
Since $b_1=1$,
the last letter of~$\beta$ is~$2$.

Let~$\lambda$ be a partition with all parts at least~$2$.
Let~$\pi$ be a type-$\lambda$-connected partition of~$T$.
Every leaf lies in the same block of~$\pi$ as its neighbor.
Connectivity then forces each block~$B$
to be a union of the vertex sets $V(G_i)$ for consecutive indices~$i$.
It is represented by a factor~$z$ of~$w(T)$ such that $\abs{z}=\abs{B}$.
We call~$z$ a \emph{block} of~$w(T)$.

Since the center~$v_\ast$ has exactly two leaves,
the internal leaf star containing~$v_\ast$ has order~$3$.
Since~$T$ is CPET,
it has a connected partition of type
\[
\lambda=
\begin{dcases*}
3^{N/3}, & if $N\equiv 0\pmod3$,
\\
4\,3^{(N-4)/3}, & if $N\equiv 1\pmod3$,
\\
3^{(N-2)/3}2, & if $N\equiv 2\pmod3$.
\end{dcases*}
\]
Let~$\pi$ be a connected partition of~$T$ of type~$\lambda$.
In the block factorization of~$w(T)$ induced by~$\pi$,
every block has size~$3$ except possibly one of size~$2$ or~$4$.

\begin{lemma}
\label{lem:adj.block3}
At least one of the side factors~$\beta$ and~$\gamma$
is a concatenation of blocks of size~$3$.
\end{lemma}

\begin{proof}
The final letter~$2$ on the $\mathbf b$-side
cannot lie in the block containing the letter~$3$,
since every block has size at most~$4$.
If that block is~$3$,
both side factors consist of whole blocks of~$\pi$.
Since~$\pi$ has at most one block of size other than~$3$,
at least one side consists entirely of blocks of size~$3$.
Otherwise the block containing the letter~$3$ is~$31$,
which is the unique block of size~$4$.
All blocks on the $\mathbf b$-side then have size~$3$.
\end{proof}

We next derive the boundary restrictions for terminal lengths $2$ and $4$.

\begin{lemma}
\label{lem:bs>=a+3}
If $a\in\{2,4\}$,
then $\mathbf b\ne \epsilon$ and $b_s\ge a+3$.
The same restrictions hold for $(d,\mathbf c)$.
\end{lemma}
\begin{proof}
For an integer $m\in\{1,\dots,N-1\}$,
we call~$m$ a \emph{$\mathbf b$-obstruction}
if no prefix of~$w(T)$ has size~$m$,
and a \emph{$\mathbf c$-obstruction}
if no prefix of~$\overline{w(T)}$ has size~$m$.
For each excluded configuration,
we exhibit a partition type for which~$T$ has no connected partition of that type.
Suppose $N=f_1+\cdots+f_m$,
where each~$f_i\ge 2$ is a $\mathbf b$-obstruction.
Then~$T$ has no connected partition of type~$f_1\dotsm f_m$.
Indeed,
the block containing~$z_1$
would be represented by a proper prefix of~$w(T)$,
but its size would be one of the $\mathbf b$-obstructions~$f_i$.
An internal leaf star of order~$q\ge 2$
following internal leaf stars of total order~$p$
gives $\mathbf b$-obstructions $p+1$,
$\dots$,
and $p+q-1$.
We treat the two parts in order.

\noindent (i).
Suppose that $a=2$.
If $\mathbf b=\epsilon$,
the center gives $\mathbf b$-obstructions~$3$ and~$4$.
By \cref{N.abcd},
we have $N\ge 9$.
The identities
\[
6=3+3,\qquad 7=3+4,\qquad 8=4+4,
\]
and increments of~$3$ cover all possible orders~$N$.

Now suppose that $\mathbf b\ne \epsilon$.
The first two internal leaf stars of order at least two
give $\mathbf b$-obstructions~$3$ and~$b_s+4$.
If $b_s=1$,
then $N\ge 9$.
The identities
\[
8=3+5,\qquad 9=3+3+3,\qquad 10=5+5,
\]
and increments of~$3$ again cover all such orders.
If $b_s=3$,
the trivalent neighbor of $v_\ast$ forces at least one further gap
on one of the two sides.
Hence $N\ge 13$ by \cref{N.abcd}.
The identities
\[
12=3+3+3+3,\qquad 13=3+3+7,\qquad 14=7+7,
\]
and increments of~$3$ cover every such order.
Since~$b_s$ is odd,
this completes part~(i).

\noindent (ii).
Suppose that $a=4$.
If $\mathbf b=\epsilon$,
the center and its trivalent neighbor give $\mathbf b$-obstructions~$5$,
$6$,
and~$8$.
Here $N\ge 11$.
The identities
\[
10=5+5,\qquad 11=5+6,\qquad 12=6+6,\qquad 13=5+8,\qquad 14=6+8,
\]
and increments of~$5$ cover all such orders.

Now suppose that $\mathbf b\ne \epsilon$ and $b_s\in\{1,3,5\}$.
The first two internal leaf stars of order at least two
give $\mathbf b$-obstructions~$5$ and $u=b_s+6\in\{7,9,11\}$.
Write
\[
S_u=\{5i+u j\colon i,j\in\mathbb Z_{\ge 0}\}.
\]
Let~$L$ and~$R$ be the weights of the $\mathbf b$-side and the $\mathbf c$-side,
respectively.
By \cref{prop:T.abcd},
both weights are even and satisfy $N=L+R+3$.
If $R=2$,
then $(d,\mathbf c)=(2,\epsilon)$,
which is excluded in (i) using $\mathbf c$-obstructions.
Hence $R\ge 4$.
The center gives $\mathbf b$-obstructions~$L+1$ and~$L+2$.
Moreover,
if a trivalent neighbor of the center lies on the $\mathbf b$-side,
then $L-1$ is a $\mathbf b$-obstruction;
if it lies on the $\mathbf c$-side,
then $L+4$ is a $\mathbf b$-obstruction.

If the chosen neighbor lies on the $\mathbf b$-side,
it suffices to show that
\[
F_1\cap S_u\ne \emptyset,
\qquad\text{where $F_1=\{R+1,\,R+2,\,R+4\}$},
\]
since
adding the corresponding $\mathbf b$-obstruction from
$L+2$,
$L+1$,
and $L-1$,
respectively,
gives~$N$.
For a neighbor on the $\mathbf c$-side,
use $F_2=\{R-1,\,R+1,\,R+2\}$,
together with $L+4$,
$L+2$,
and $L+1$,
respectively.
\Cref{tab:boundary.adj} supplies a suitable element
for $R\in\{4,6,8,10,12\}$.
\begin{table}[H]
\centering
\caption{Sums of $5$ and $u$ completing a $\mathbf b$-obstruction to~$N$.}
\label{tab:boundary.adj}
\begin{tabular}{@{}l*{6}{>{$}c<{$}}@{}}
\toprule
Neighbor on & u & R=4 & R=6 & R=8 & R=10 & R=12
\\
\midrule
$\mathbf b$-side & 7,9 & 5 & 10 & 10 & 14 & 14
\\
$\mathbf b$-side & 11 & 5 & 10 & 10 & 11 & 16
\\
$\mathbf c$-side & 7 & 5 & 5 & 10 & 12 & 14
\\
$\mathbf c$-side & 9 & 5 & 5 & 10 & 9 & 14
\\
$\mathbf c$-side & 11 & 5 & 5 & 10 & 11 & 11
\\
\bottomrule
\end{tabular}
\end{table}
Here each entry in the last~$5$ columns is a sum of parts in $\{5,u\}$;
it is in~$F_1$ for the first~$2$ rows,
and in~$F_2$ for the last~$3$ rows.
For instance,
the entry $14=7+7=5+9$.
Every even $R\ge 4$ is one of the listed values
plus a nonnegative multiple of $10=5+5$.
Adding that multiple to the corresponding entry
preserves both membership in~$S_u$ and its required form.
The table and its extensions therefore exclude all three values of~$b_s$,
completing part~(ii).

Reflection gives the same conclusions for~$(d,\mathbf c)$
using $\mathbf c$-obstructions.
\end{proof}

We now apply the framework of \cref{sec:framework}
with $m=10$ and $\theta=\psi$,
where the nonzero values of~$\psi$ on the elementary generators are
\[
\psi(e_2)=1,\quad
\psi(e_3)=\frac{251}{5000},\quad
\psi(e_4)=\frac{203}{400},\quad
\psi(e_6)=\frac{163}{1000},\quad
\psi(e_8)=\frac{73}{2000},\quad
\psi(e_{10})=\frac{49}{10000}.
\]
Define a \emph{monoid homomorphism}~$h$ from the free monoid on~$\{1,2\}$
to the multiplicative monoid of rational $10\times10$ matrices by
$h(1)=A$ and $h(2)=M_1$.
By \cref{def:A},
\begingroup
\renewcommand{\arraystretch}{1.8}
\setlength{\arraycolsep}{8pt}
\[
A=\brk*{
\begin{array}{c|ccccccccc}
\mathbf 0&\multicolumn{9}{c}{I_9}
\\
\hline
\rule{0pt}{2em}\dfrac{441}{10000}&0&\dfrac{511}{2000}&0&\dfrac{163}{200}&0&
\dfrac{609}{400}&\dfrac{251}{2500}&1&0
\end{array}
}.
\]
\endgroup
For any word~$z$ over~$\{1,2\}$,
define its \emph{evaluation row} by
\[
r(z)=\mathbf e_1^\top A^{-1}h(z).
\]
Then for any words~$z$ and~$z'$ over~$\{1,2\}$,
the homomorphism property of~$h$ gives
\begin{equation}
\label{r=rh}
r(zz')
=\mathbf e_1^\top A^{-1}h(z)h(z')
=r(z)h(z').
\end{equation}

Let
\begin{align*}
\delta&=r(1^6)=\mathbf e_1^\top A^{-1}A^6
=\mathbf e_6^\top=(0,0,0,0,0,1,0,0,0,0),
\\
\varepsilon&=r(1^2 2 1^4)
=\mathbf e_1^\top A^{-1}A^2 (A^2-\ell \mathbf e_1^\top) A^4
=\mathbf e_8^\top-L_3\mathbf e_5^\top
=(0,0,0,0,-753/5000,0,0,1,0,0),
\quad\text{and}\\
\zeta&=r(1^4 2 1^3)
=\mathbf e_1^\top A^{-1} A^4 (A^2-\ell \mathbf e_1^\top) A^3
=\mathbf e_9^\top-L_5\mathbf e_4^\top
=(0,0,0,-1757/5000,0,0,0,0,1,0).
\end{align*}
Let $J$ be the $30\times 10$ matrix
\[
J=
\begin{pmatrix}
\delta
\\
HA^4
\\
HA^2M_1
\\
\zeta A^3
\end{pmatrix}, 
\]
where~$H$ is the $14\times10$ matrix
\begingroup
\setlength{\arraycolsep}{1pt}
\[
\begin{pmatrix}
0&0&0&0&0&1&0&0&0&0
\\
804715&21168&-145196&132893&-10000000&455358&1443927&1033951&1592797&985794
\\
-482912&21093&-159004&132467&-10000000&454090&1445511&1031974&1598166&982388
\\
-61577&4099&108538&25835&111221&89618&-1000000&218147&417586&193215
\\
-342703&9622&-10000000&60398&173930&206906&675699&473513&775654&449285
\\
-53391&3264&-33054&18396&-33823&57097&-15830&100000&7841&50609
\\
-10000000&5176&49102&32519&162939&111190&339117&249997&348319&239714
\\
-2848&-17&-554&-107&-702&-365&-528&10000&192&-763
\\
-640912&14083&-232094&105916&294730&396990&-234569&645095&-97570&1000000
\\
1000000&2102&732175&13374&426832&45822&13907&78342&-214733&87389
\\
-1000000&2353&681216&14939&385935&51779&48315&92287&-165771&99596
\\
-13489&386&11932&2425&100000&8404&43637&12984&-50731&15917
\\
-3528&-52&-3486&-327&-2130&-1413&-804&-2435&197&10000
\\
254158&148934&1378435&935758&4663152&3203782&9732753&7201248&10000000&6901416
\end{pmatrix}.
\]
\endgroup

\begin{lemma}
\label{lem:coneH.adj}
We have
$\delta,\varepsilon,\zeta A\in\cone{H}$,
$\cone{H} A^2\subseteq\cone{H}$,
$\cone{H} M_1\subseteq\cone{H}$,
and $HW_\psi J^\top<0$.
\end{lemma}

\begin{proof}
Since $\delta=\mathbf e_6^\top$ is the first row of~$H$,
we have $\delta\in\cone{H}$.
The self-contained script \texttt{verify.py} constructs
nonnegative rational rows~$u$ and~$v$
and nonnegative rational matrices~$U$ and~$V$ satisfying
\[
\varepsilon=uH,
\quad
\zeta A=vH,
\quad
HA^2=UH,
\quad\text{and}\quad
HM_1=VH.
\]
For each of these~$1+1+14+14=30$ row identities,
the script specifies ten rows of~$H$,
checks their independence,
and solves for the coefficients using exact rational elimination.
It verifies both the identity and the nonnegativity of every coefficient.
The first two identities give $\varepsilon,\zeta A\in\cone{H}$,
and the last two give the invariance of~$\cone{H}$ by linearity.
Finally,
the script constructs~$W_\psi$ and~$J$ from their definitions
and verifies that all $14\cdot30=420$ entries of~$HW_\psi J^\top$ are negative.
\end{proof}

We now show that both side evaluation rows belong to~$\cone{H}$,
and that a side consisting of blocks of size~$3$ gives a row in~$\cone{J}$.

\begin{lemma}
\label{lem:cone.side.adj}
Both $r(\beta)$ and~$r(\gamma)$ belong to~$\cone{H}$.
\end{lemma}

\begin{proof}
Let $(\alpha,a_0)$ be either $(\beta,a)$ or $(\gamma,d)$.
By \cref{beta:gamma,lem:bs>=a+3},
we may choose a prefix~$z$ of~$\alpha$ as follows.
If $a_0=2$,
take $z=1^2 2 1^4$; then $r(z)=\varepsilon$.
If $a_0=4$,
take $z=1^4 2 1^4$;
then $r(z)=\zeta A$ by \cref{r=rh}.
If $a_0\ge 6$,
take $z=1^6$; then $r(z)=\delta$.
In each case,
$r(z)\in\cone{H}$ by \cref{lem:coneH.adj}.

Write $\alpha=z\sigma$.
Every run of ones in~$\alpha$ has even length
by \cref{prop:T.abcd}.
Since~$z$ ends after an even number of ones in its final run,
every run of ones in~$\sigma$ also has even length.
Pairing these ones within each run and keeping each letter~$2$ separately
expresses~$\sigma$ as a possibly empty concatenation of~$11$ and~$2$.
By \cref{r=rh},
appending these words multiplies the evaluation row by~$A^2$ and~$M_1$,
respectively.
The two inclusions in \cref{lem:coneH.adj}
therefore give $r(\alpha)\in\cone{H}$.
\end{proof}

\begin{lemma}
\label{lem:coneJ.adj}
If a side factor~$\alpha$ is a concatenation of blocks of size~$3$,
then $r(\alpha)\in\cone{J}$.
\end{lemma}

\begin{proof}
Let~$a_0$ be the terminal length of~$\alpha$.
Since~$\abs{\alpha}$ is even and every block has size~$3$,
the size~$\abs{\alpha}$ is a multiple of~$6$.
If $a_0=2$,
then~$\alpha$ begins with~$112$
by \cref{lem:bs>=a+3},
and its first block does not have size~$3$,
a contradiction.
Thus $a_0\ge 4$.
If $\alpha=1^6$ or $\alpha=1^4 2 1^6$,
then its evaluation row is~$\delta$ or~$\zeta A^3$,
respectively,
by \cref{r=rh},
and is a row of~$J$.
We may therefore exclude these two words.

By \cref{prop:T.abcd},
there exists a prefix~$\tau$ of even size $\abs{\alpha}-4$ such that
$\alpha=\tau1111$ or $\alpha=\tau112$.
If $a_0=4$,
then~$\alpha$ begins with~$1^4 2 1^6$
by \cref{lem:bs>=a+3}.
Since this word was excluded and~$\abs{\alpha}$ is a multiple of~$6$,
we have $\abs{\alpha}\ge 18$.
Thus~$\tau$ contains the prefix~$1^4 2 1^4$,
whose evaluation row is~$\zeta A\in\cone{H}$
by \cref{r=rh}
and \cref{lem:coneH.adj}.
If $a_0\ge 6$,
then $\abs{\alpha}\ge 12$;
otherwise $\abs{\alpha}=6=a_0$ and $\alpha=1^6$,
which was excluded.
Thus~$\tau$ contains the prefix~$1^6$,
whose evaluation row is~$\delta\in\cone{H}$.
In either case,
every remaining run of ones in~$\tau$ has even length,
as in the proof of \cref{lem:cone.side.adj}.
Appending~$11$ and~$2$ therefore gives $r(\tau)\in\cone{H}$
by \cref{r=rh}
and the two invariance inclusions in \cref{lem:coneH.adj}.
Applying \cref{r=rh} to the final suffix gives
$r(\alpha)\in\cone{HA^4}$
or
$r(\alpha)\in\cone{HA^2M_1}$.
Both cones are contained in~$\cone{J}$ by the definition of~$J$,
completing the proof.
\end{proof}

We can now prove \cref{prop:adj}.

\begin{proof}
By \cref{prop:T.abcd} and \cref{prop:Wolfgang},
it suffices to show that every CPET caterpillar
$T=T(a;\mathbf b\mid \mathbf c;d)$
whose center has a trivalent neighbor is not $e$-positive.
By symmetry,
we may assume $b_1=1$.
Since~$h$ is a homomorphism,
$h(2 1^{j-1})=M_1 A^{j-1}=M_j$ for $j\ge 1$.
By \cref{beta:gamma}
and \cref{r=rh},
\[
r(\beta)
=\mathbf e_1^\top A^{a-1}M_{\mathbf b}
=r_{a,\mathbf b},
\]
and similarly
$r(\gamma)=r_{d,\mathbf c}$.
By \cref{lem:cone.side.adj},
both $r(\beta)$ and~$r(\gamma)$ belong to~$\cone{H}$.
By \cref{lem:adj.block3,lem:coneJ.adj},
at least one of them belongs to~$\cone{J}$.
Both rows are nonzero by \cref{lem:inv}.
By the symmetry of~$W_\psi$,
we may interchange~$\beta$ and~$\gamma$ if necessary
and write $r(\beta)=\lambda H$ and $r(\gamma)=\mu J$
for nonzero nonnegative coefficient rows~$\lambda$ and~$\mu$.
Thus \cref{lem:transfer,lem:coneH.adj} give
\[
\psi(X_T)
=r(\beta)W_\psi r(\gamma)^\top
=\lambda(HW_\psi J^\top)\mu^\top<0.
\]
Since $\psi(e_j)\ge 0$ for every $j\ge 1$,
$T$ is not $e$-positive.
\end{proof}

Combining \cref{prop:nonadj,prop:adj}
proves \cref{thm:e-.caterpillar}.

\section{Two infinite CPET families}
\label{sec:CPET}

We give two infinite families of non-$e$-positive CPET caterpillars
in the nonadjacent and adjacent cases,
respectively.
These families show that \cref{prop:Wolfgang} alone cannot establish
\cref{thm:e-.caterpillar} in either case.
For the adjacent family,
we determine exactly which parameter values give CPET caterpillars.
Put
$G_a=T(a;(1)\,\vert\,\epsilon;a-2)$
and $N=2a+3$.

\begin{lemma}
\label{lem:CPET.mu.Ga}
Let~$a\ge 4$ be even.
Let $\mu=\mu_1\dotsm\mu_\ell\vdash N$
with $\mu_i\ge 2$ for $1\le i\le \ell$.
Then~$G_a$ has a connected partition of type~$\mu$
if and only if at least one of the following conditions holds:
\begin{enumerate}[label=(\roman*)]
\item
\label[itm]{itm:adjacent-cpet-separated}
There is a set~$I\subseteq[\ell]$ such that
$\sum_{i\in I}\mu_i\in\{a+1,a+2\}$.
\item
\label[itm]{itm:adjacent-cpet-together}
There are an index~$k\in[\ell]$ and disjoint sets~$I$ and~$J$
such that
\[
I\sqcup J=[\ell]\setminus\{k\},
\qquad
\sum_{i\in I}\mu_i\le a,
\qquad\text{and}\qquad
\sum_{i\in J}\mu_i\le a-2.
\]
\end{enumerate}
\end{lemma}

\begin{proof}
Let~$u$ be the trivalent vertex of~$G_a$.
It is adjacent to the center~$v_\ast$.
As before,
since $\mu_i\ge 2$,
every internal leaf star is contained in a block
of every connected partition of type~$\mu$.
If~$u$ and~$v_\ast$ lie in different blocks,
then each block lies in a component of $G_a-u v_\ast$,
whose orders are $a+2$ and~$a+1$.
The parts assigned to either component therefore satisfy (i).
Otherwise,
deleting the two internal leaf stars
leaves two paths~$P_a$ and~$P_{a-2}$.
Every other block lies in one of these paths,
which gives (ii).

Conversely,
suppose that (i) holds.
Suppose that $\sum_{i\in I}\mu_i=a+2$ without loss of generality.
Then
\[
G_a-uv_\ast\cong P_{a+2}\sqcup S(a-2,\,1,\,1).
\]
Partition the path consecutively using the parts indexed by~$I$.
Since the remaining parts sum to the odd number~$a+1$,
some part~$q$ is odd and at least~$3$.
Constructing this block to contain the leaf star at the center,
and then partitioning the remaining path consecutively gives the desired partition.

If (ii) holds,
remove consecutive partition blocks of the sizes indexed by~$I$ and~$J$
from the terminal ends of the arms of lengths~$a$ and~$a-2$,
respectively.
The two bounds ensure that the remainder is connected and has size~$\mu_k$,
completing the partition.
\end{proof}

The following proposition characterizes exactly when~$G_a$ is CPET.

\begin{proposition}
\label{prop:CPET.Ga}
Let~$a\ge 3$ be an integer.
Then the caterpillar~$G_a$ is CPET if and only if
\[
a\equiv 2\pmod4
\qquad\text{and}\qquad
\gcd(2a+3,\,15)=1.
\]
\end{proposition}

\begin{proof}
Suppose that~$G_a$ is CPET.
By \cref{prop:T.abcd},
$a$ is even.
If $4\mid a$,
consider the type $3\,4^{a/2}$.
Every sum in condition~\cref{itm:adjacent-cpet-separated}
of \cref{lem:CPET.mu.Ga}
is congruent to~$0$ or~$3$ modulo~$4$,
so that condition fails.
Since all parts are at most $4$,
condition~\cref{itm:adjacent-cpet-together} also fails
since $\mu_k\ge N-a-(a-2)=5$.
Hence $a\equiv 2\pmod4$.

If $3\mid N$,
then $3\mid a$.
The type~$3^{N/3}$ satisfies neither
condition~\cref{itm:adjacent-cpet-separated}
nor condition~\cref{itm:adjacent-cpet-together}:
neither $a+1$ nor $a+2$ is divisible by~$3$,
and every part is smaller than~$5$.
Hence $3\nmid N$.
If $5\mid N$,
then $a\equiv 1\pmod5$.
For the type~$5^{N/5}$,
condition~\cref{itm:adjacent-cpet-separated} fails
because neither $a+1$ nor $a+2$ is divisible by~$5$.
In condition~\cref{itm:adjacent-cpet-together},
the distinguished block must have size~$5$,
and both upper bounds are attained.
It follows that both~$a$ and $a-2$ are multiples of~$5$,
which is impossible.
Thus $5\nmid N$.
Hence $\gcd(N,15)=1$.

Conversely,
suppose that the two conditions hold.
Then $a\ge 6$ is even.
Let $\mu\vdash N$.
The type $\mu=1^N$ is realized by the singleton partition.
Otherwise,
merge all parts~$1$ into one part of size at least~$2$.
Once a connected partition of the resulting type is constructed,
recover the singleton blocks by successively removing leaves
from the tree induced by the merged block;
the remaining block stays connected.
It therefore suffices to consider partitions~$\mu$ with all parts at least~$2$.

Set aside the largest part~$\mu_1$.
Starting with two empty groups,
place each remaining part in the group with the smaller current sum,
breaking ties arbitrarily.
Since each part is at most~$\mu_1$,
the absolute difference between the two sums remains at most~$\mu_1$.
Denote the final sums by~$x\le y$.
Then $x+y=N-\mu_1$ and $y-x\le \mu_1$.
It follows that $2y\le N=2a+3$ and $y\le a+1$.
If $y=a+1$,
then condition~\cref{itm:adjacent-cpet-separated} of
\cref{lem:CPET.mu.Ga} applies.
Otherwise $y\le a$.
If $x\le a-2$,
then the groups of sums~$y$ and~$x$ together with the part~$\mu_1$
satisfy condition~\cref{itm:adjacent-cpet-together}.
It remains to consider the pairs $(x,y)$ such that
\[
a-1\le x\le y\le a,
\qquad\text{i.e.,}\qquad
(x,y)\in\{(a,a),\ (a-1,\,a),\ (a-1,\,a-1)\}.
\]

If $x=y=a$,
then $\mu_1=N-x-y=3$.
Since $3\nmid N$,
we find $\mu_\ell=2$.
Moving~$\mu_\ell$ to the other group gives a sum of~$a+2$,
so condition~\cref{itm:adjacent-cpet-separated} applies.
If $(x,y)=(a-1,\,a)$,
then $\mu_1=4$.
Since $a\equiv 2\pmod4$,
the group of sum~$a$ contains a part~$2$ or~$3$.
Moving this part to the other group gives a sum of~$a+1$ or~$a+2$,
and condition~\cref{itm:adjacent-cpet-separated} applies.
If $x=y=a-1$,
then $\mu_1=5$.
If either group contains a part~$2$ or~$3$,
then moving it to the other group gives a sum of~$a+1$ or~$a+2$,
and condition~\cref{itm:adjacent-cpet-separated} applies.
Since $5\nmid N$,
we find $\mu_\ell=4$.
Since $a-1$ is odd,
each group contains a part~$5$.
Exchanging~$\mu_\ell$ with a part~$5$ in the other group yields
two sums~$a$ and~$a-2$,
and condition~\cref{itm:adjacent-cpet-together} applies.
\end{proof}

We now construct the aforementioned CPET caterpillars.

\begin{corollary}
\label{cor:infinite.CPET}
The caterpillars $T(3p-3;\epsilon\,\vert\,(3p-4);2)$ for primes~$p\ge 7$
and $T(a;(1)\,\vert\,\epsilon;a-2)$ for integers~$a\ge 6$
with $a\equiv 2\pmod4$ and $\gcd(2a+3,15)=1$
form two infinite families of non-spider CPET caterpillars that are not $e$-positive.
\end{corollary}

\begin{proof}
The two caterpillars are not $e$-positive by \cref{thm:e-.caterpillar}.
\citet[Prop.~9 and Thm.~10]{CGMPW06} proved that the first one is CPET,
and \cref{prop:CPET.Ga} shows that~$T(a;(1)\,\vert\,\epsilon;a-2)$ is CPET.
To see the infinitude of the second family,
the choice $a=60q+10$ satisfies both conditions for every integer $q\ge 0$.
\end{proof}

The caterpillars in \cref{cor:infinite.CPET} are shown in \cref{fig:Tp}.
\begin{figure}[H]
\centering
\begin{tikzpicture}[
x=1.0cm, y=0.84cm,
every node/.style={font=\small}]
\draw[graph edge] (0,0)--(1.7,0);
\draw[graph edge] (2.3,0)--(4,0)--(5.7,0);
\draw[graph edge] (6.3,0)--(8,0)--(9,0)--(10,0);
\draw[graph edge] (3.58,0.92)--(4,0)--(4.42,0.92);
\draw[graph edge] (8,0)--(8,0.92);
\foreach \p in {(0,0),(4,0),(8,0),(9,0),(10,0),(3.58,0.92),(4.42,0.92),(8,0.92)}
\node[graph vertex] at \p {};
\node[inner sep=1.2pt] at (2,0) {$\cdots$};
\node[inner sep=1.2pt] at (6,0) {$\cdots$};
\draw[segment brace]
(0.06,-0.62)--node[below=4pt] {$3p-3$} (3.94,-0.62);
\draw[segment brace]
(4.06,-0.62)--node[below=4pt] {$3p-4$} (7.94,-0.62);
\draw[segment brace]
(8.06,-0.62)--node[below=4pt] {$2$} (9.94,-0.62);
\begin{scope}[yshift=-2.3cm]
\draw[graph edge] (0,0)--(1.7,0);
\draw[graph edge] (2.3,0)--(4,0)--(5,0)--(7.2,0);
\draw[graph edge] (7.8,0)--(10,0);
\draw[graph edge] (4,0)--(4,0.92);
\draw[graph edge] (4.58,0.92)--(5,0)--(5.42,0.92);
\foreach \p in {(0,0),(4,0),(5,0),(10,0),(4,0.92),(4.58,0.92),(5.42,0.92)}
\node[graph vertex] at \p {};
\node[inner sep=1.2pt] at (2,0) {$\cdots$};
\node[inner sep=1.2pt] at (7.5,0) {$\cdots$};
\draw[segment brace]
(0.06,-0.62)--node[below=4pt] {$a$} (3.94,-0.62);
\draw[segment brace]
(4.06,-0.62)--node[below=4pt] {$1$} (4.94,-0.62);
\draw[segment brace]
(5.06,-0.62)--node[below=4pt] {$a-2$} (9.94,-0.62);
\end{scope}
\end{tikzpicture}
\caption{The two caterpillars in \cref{cor:infinite.CPET}.
The labels below the horizontal paths indicate the segment lengths.}
\label{fig:Tp}
\end{figure}
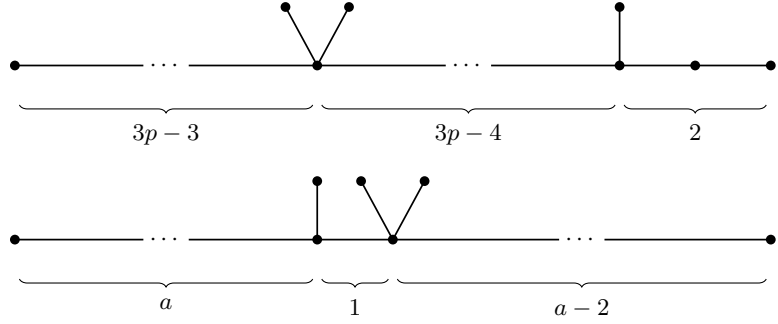

\section{Explicit examples and a criterion for caterpillars with $\Delta=3$}
\label{sec:Li}

Denote by
$C(\alpha)=C(\alpha_1,\dots,\alpha_d)$ the caterpillar
whose nonleaf vertices form the path~$v_1\dotsm v_d$,
with~$\alpha_i$ leaves adjacent to~$v_i$.
\citet[Thm.~3.5]{Li26X} established the following criterion for $e$-positivity of caterpillars.

\begin{theorem}[\citeauthor{Li26X}]
\label{thm:Li.caterpillar}
For any $e$-positive caterpillar $C(\alpha)=C(\alpha_1,\ldots,\alpha_d)$, 
\begin{multline*}
\sum_{i=1}^d\binom{\alpha_i}{2}
+\sum_{i=1}^d\alpha_i\sum_{j\ne i}\abs{j-i}
+\sum_{1\le i<j\le d}\brk1{\alpha_i\alpha_j(j-i+1)+j-i-1}
\\
\ge\sum_{k=\lceil\frac{\abs{\alpha}+d}{2}\rceil}^{\abs{\alpha}+d-2}
k(\abs{\alpha}+d-k)
+\frac{(\abs{\alpha}-d)(\abs{\alpha}+d-1)}{2}.
\end{multline*}
\end{theorem}

An example that passes the inequality in \cref{thm:Li.caterpillar} is the CPET
caterpillar
\[
T_2=T(4;(3)\,\vert\,\epsilon;6)
=C(1,0,0,1,0,0,2,0,0,0,0,1),
\]
shown in \cref{fig:Li.examples}; see also~\citet[Fig.~4.1]{Li26X}.

\begin{figure}[ht]
\centering
\begin{tikzpicture}[x=0.6cm, y=0.8cm]
\draw[graph edge] (0,0)--(13,0);
\draw[graph edge] (4,0)--(4,0.9)
(6.5,0.9)--(7,0)--(7.5,0.9);
\foreach \i in {0,...,13}
\node[graph vertex] at (\i,0) {};
\foreach \p in {(4,0.9),(6.5,0.9),(7.5,0.9)}
\node[graph vertex] at \p {};
\node[below=4pt] at (4,0) {$u_1$};
\node[below=4pt] at (7,0) {$v_\ast$};
\end{tikzpicture}
\caption{The CPET caterpillar~$T_2$
which satisfies the inequality in \cref{thm:Li.caterpillar}.}
\label{fig:Li.examples}
\end{figure}

We use~$T_2$ to illustrate the specialization--cone method.
By \cref{lem:cone.nonadj},
the seed row~$\mathbf e_2^\top$ belongs to~$\cone{F_\varphi}$,
which is invariant under~$D_0=A^2$ and~$D_1=M_1 A^2$.
For the two sides of~$T_2$,
the transfer operations give
\[
r_{4,(3)}
=\mathbf e_2^\top D_0 D_1
=(143,46,71,38)F_\varphi/196
\qquad\text{and}\qquad
r_{6,\epsilon}
=\mathbf e_2^\top D_0^2
=(193,549,335,200)F_\varphi/784.
\]
Thus both side evaluation rows belong to~$\cone{F_\varphi}$.
Let~$\lambda$ and~$\mu$ denote the respective positive coefficient rows above.
By \cref{lem:transfer,lem:cone.nonadj},
\[
\varphi(X_{T_2})
=r_{4,(3)}W_\varphi r_{6,\epsilon}^\top
=\lambda(F_\varphi W_\varphi F_\varphi^\top)\mu^\top<0.
\]
Since $\varphi(e_j)\ge 0$ for every~$j\ge 1$,
this proves that~$T_2$ is not $e$-positive.

For the adjacent case in \cref{prop:adj},
take the smallest CPET member of the family in \cref{prop:CPET.Ga},
namely
\[
G_{10}=T(10;(1)\,\vert\,\epsilon;8)
=C(1,0^8,1,2,0^6,1),
\]
shown in \cref{fig:Li.adjacent}.
\cref{thm:Li.caterpillar} does not work for it since the inequality in reads $1449\ge 847$.

\begin{figure}[ht]
\centering
\begin{tikzpicture}[x=0.6cm, y=0.8cm]
\draw[graph edge] (0,0)--(19,0);
\draw[graph edge] (10,0)--(10,0.9)
(10.6,0.9)--(11,0)--(11.4,0.9);
\foreach \i in {0,...,19}
\node[graph vertex] at (\i,0) {};
\foreach \p in {(10,0.9),(10.6,0.9),(11.4,0.9)}
\node[graph vertex] at \p {};
\node[below=4pt] at (10,0) {$u_1$};
\node[below=4pt] at (11,0) {$v_\ast$};
\end{tikzpicture}
\caption{The CPET caterpillar~$G_{10}$
which satisfies the inequality in \cref{thm:Li.caterpillar}.}
\label{fig:Li.adjacent}
\end{figure}

We use~$G_{10}$ to illustrate the specialization--cone method.
The side factors are $\beta=1^{10}2$ and $\gamma=1^8$.
A connected partition of the type~$3^7 2$ used in the proof of \cref{prop:adj}
is given by
\[
w(G_{10})
=\underbrace{111\mid111\mid111\mid12}_{\beta}
\mid3\mid
\underbrace{111\mid111\mid11}_{\gamma}.
\]
In particular,
$\beta$ consists of four blocks of size~$3$.
By \cref{lem:coneH.adj},
$\delta\in\cone{H}$ and $\cone{H}A^2\subseteq\cone{H}$.
Thus there is a nonzero nonnegative row~$\nu$ such that
\[
r_{8,\epsilon}
=\delta A^2
=\mathbf e_8^\top
=\nu H
\qquad\text{and}\qquad
r_{10,(1)}
=r_{8,\epsilon}A^2M_1
=\nu HA^2M_1.
\]
Since~$HA^2M_1$ is a row block of~$J$,
we have $r_{10,(1)}\in\cone{J}$
and $HW_\psi(HA^2M_1)^\top<0$
by \cref{lem:coneH.adj}.
By \cref{lem:transfer} and the symmetry of~$W_\psi$,
\[
\psi(X_{G_{10}})
=r_{8,\epsilon}W_\psi r_{10,(1)}^\top
=\nu\brk[s]1{HW_\psi(HA^2M_1)^\top}\nu^\top<0.
\]
Explicit computation gives $\psi(X_{G_{10}})<-2$ indeed. 
Since $\psi(e_j)\ge 0$ for every~$j\ge 1$,
this proves that~$G_{10}$ is not $e$-positive.

\cref{thm:e-.caterpillar} reduces
\cref{thm:Li.caterpillar} to the following criterion.

\begin{corollary}
\label{cor:Delta=3.caterpillar}
Let $C(\alpha_1,\dots,\alpha_d)$ be an $e$-positive caterpillar
on~$n$ vertices with $d\ge 2$ nonleaf vertices, where
$\alpha_1,\alpha_d\in\{1,2\}$,
and $\alpha_i\in\{0,1\}$
for $2\le i\le d-1$.
Let~$v_{i_1},\dots,v_{i_r}$ be its trivalent vertices,
listed in increasing order of their indices,
where~$r\ge 0$.
Then $n=d+r+2$,
and
\[
\sum_{j=1}^r(i_j+j)(n-i_j-j)
\le
\dfrac{(n-1)(n-3)(n+4)}{12}
-\dfrac{n^2}{8}\mathbf 1(\text{$n$ is even}).
\]
\end{corollary}

\begin{proof}
The result follows from \cref{thm:e-.caterpillar,thm:Li.caterpillar}
and a routine algebraic computation.
\end{proof}

\Cref{cor:Delta=3.caterpillar} is still a restriction when $\Delta=3$.
For example, the double broom $C(2,2)$ has $n=6$, $r=2$, and $(i_1,i_2)=(1,2)$,
but the two sides of the inequality above satisfy $2(6-2)+4(6-4)>8$.

\section*{Disclosure of artificial intelligence use}
\label{sec:AI}

We used successive OpenAI models,
including GPT-5.5,
GPT-5.6,
and GPT-6 Astra,
as well as OpenAI Codex, 
for mathematical exploration,
computation,
and manuscript preparation.

GPT-5.5 developed a specialization method establishing non-$e$-positivity
for some caterpillars with at most three trivalent vertices
in the nonadjacent case.
This initiated the specialization approach to
\cref{thm:e-.caterpillar}.
GPT-5.6 then completed the nonadjacent case.
It raised the idea of the specialization--cone method, 
and supplied explicit specializations,
cone generators,
and finite certificates.
The model later developed a proof of \cref{prop:adj}
using four specializations.
GPT-6 Astra developed the coloring proof of \cref{lem:transfer},
found the specialization~$\varphi$ and the cone~$\cone{F_\varphi}$ in \cref{sec:nonadj},
and successively simplified the proof of \cref{prop:adj}.
The coloring proof replaced the earlier argument using rooted pendant extensions.
An earlier version took the last entry of the second row of~$F_\varphi$ to be~$0$, 
causing that proof of \cref{prop:nonadj} twice long.
An intermediate proof of \cref{prop:adj} used a single specialization
with larger finite support and a four-cone certificate.
The current two-cone certificate reduces the number 
of strict entrywise matrix inequalities to be verified
from~$927,624$ to~$420$
and the number of exact row identities from~$6,192$ to~$30$.

Codex wrote and revised the self-contained verifier~\texttt{verify.py},
which checks the certificate in \cref{lem:coneH.adj}
using exact rational arithmetic.
It also assisted with the organization of the paper,
reference checking,
wording,
notation,
and \LaTeX\ source preparation.

The authors initiated the project,
derived the structural reduction in \cref{prop:T.abcd},
proposed studying $e$-coefficients indexed by partitions with parts~$2$ and~$3$,
and directed the division into the nonadjacent and adjacent cases.
They guided the mathematical exploration such like \cref{cor:infinite.CPET}, 
and simplification throughout,
independently verified every result,
and checked every deduction.
They take full responsibility for the accuracy and content of this paper.

\bibliography{csf-caterpillar}

@article{BF06,
  author =        {Barth, D. and Fournier, H.},
  journal =       {Discrete Math.},
  number =        {5},
  pages =         {469--477},
  title =         {A degree bound on decomposable trees},
  volume =        {306},
  year =          {2006},
}

@article{Bir67IC,
  author =        {Garrett Birkhoff},
  journal =       {Amer. Math. Monthly},
  number =        {3},
  pages =         {274--276},
  title =         {Linear transformations with invariant cones},
  volume =        {74},
  year =          {1967},
  doi =           {10.2307/2316020},
}

@article{BC18,
  author =        {Brosnan, Patrick and Chow, {\relax T.Y}},
  journal =       {Adv. Math.},
  pages =         {955--1001},
  title =         {Unit interval orders and the dot action on the
                   cohomology of regular semisimple {H}essenberg
                   varieties},
  volume =        {329},
  year =          {2018},
}

@article{CC24,
  author =        {Chan, William and Crew, Logan},
  journal =       {J. Graph Theory},
  number =        {4},
  pages =         {633--651},
  title =         {A graph polynomial from chromatic symmetric
                   functions},
  volume =        {105},
  year =          {2024},
}

@article{CHW26,
  author =        {Chen, {\relax L} and He, {\relax Y.T} and
                   Wang, {\relax D.G.L}},
  journal =       {Discrete Math.},
  pages =         {114723},
  title =         {Clocks are $e$-positive},
  volume =        {349},
  number =        {1},
  year =          {2026},
}

@article{CGMPW06,
  author =        {S. Cichacz and A. G{\"o}rlich and A. Marczyk and
                   J. Przyby{\l}o and M. Wo{\'z}niak},
  journal =       {Discuss. Math. Graph Theory},
  number =        {2},
  pages =         {291--305},
  title =         {Arbitrarily vertex decomposable caterpillars with
                   four or five leaves},
  volume =        {26},
  year =          {2006},
}

@article{CHS07,
  author =        {David Coudert and Florian Huc and
                   Jean-S{\'e}bastien Sereni},
  journal =       {J. Graph Theory},
  number =        {1},
  pages =         {27--41},
  title =         {Pathwidth of outerplanar graphs},
  volume =        {55},
  year =          {2007},
}

@article{DFv20,
  author =        {Samantha Dahlberg and Ang\`ele Foley and
                   Stephanie {van Willigenburg}},
  journal =       {J. Eur. Math. Soc. (JEMS)},
  number =        {8},
  pages =         {2673--2696},
  publisher =     {European Mathematical Society Publishing House},
  title =         {Resolving {S}tanley's $e$-positivity of
                   claw-contractible-free graphs},
  volume =        {22},
  year =          {2020},
}

@article{DSv20,
  author =        {Dahlberg, {\relax S} and She, {\relax A} and
                   {van Willigenburg}, {\relax S}},
  journal =       {Electron. J. Combin.},
  number =        {1},
  pages =         {\#P1.2},
  title =         {Schur and $e$-Positivity of Trees and Cut Vertices},
  volume =        {27},
  year =          {2020},
}

@article{FHM19,
  author =        {Foley, {\relax A.M} and Ho{\`a}ng, {\relax C.T} and
                   Merkel, {\relax O.D}},
  journal =       {Electron. J. Combin.},
  number =        {3},
  pages =         {\#P3.51},
  title =         {Classes of graphs with {$e$}-positive chromatic
                   symmetric function},
  volume =        {26},
  year =          {2019},
}

@article{GNRS24,
  author =        {Ganesan, A. and Narayanan, N. and Rao, {\relax B.V.R} and
                   Sawant, {\relax S.S}},
  journal =       {Discrete Math.},
  note =          {8 pp.},
  number =        {11},
  pages =         {114162},
  title =         {Proper {$q$}-caterpillars are distinguished by their
                   chromatic symmetric functions},
  volume =        {347},
  year =          {2024},
}

@article{Gas96,
  author =        {Gasharov, Vesselin},
  booktitle =     {Proceedings of the 6th {C}onference on {F}ormal
                   {P}ower {S}eries and {A}lgebraic {C}ombinatorics
                   ({N}ew {B}runswick, {NJ}, 1994)},
  journal =       {Discrete Math.},
  number =        {1--3},
  pages =         {193--197},
  title =         {Incomparability graphs of {$(3+1)$}-free posets are
                   {$s$}-positive},
  volume =        {157},
  year =          {1996},
}

@article{Hik24X,
  author =        {Hikita, {\relax T}},
  journal =       {J. Amer. Math. Soc.},
  note =          {Early View},
  title =         {A proof of the {Stanley--Stembridge} conjecture},
  year =          {2026},
  url =           {https://doi.org/10.1090/jams/1090},
}

@book{HJ13B,
  author =        {Horn, {\relax R.A} and Johnson, {\relax C.R}},
  edition =       {2nd},
  publisher =     {Camb. Univ. Press, Cambridge},
  title =         {Matrix Analysis},
  year =          {2013},
}

@article{HW08,
  author =        {M. Hor{\v{n}}{\'a}k and M. Wo{\'z}niak},
  journal =       {Discrete Math.},
  number =        {7},
  pages =         {1268--1281},
  title =         {On arbitrarily vertex decomposable trees},
  volume =        {308},
  year =          {2008},
}

@article{HJLOY21,
  author =        {B.-H. Hwang and W.-S. Jung and K.-J. Lee and J. Oh and
                   S.-H. Yu},
  journal =       {J. Combin. Theory Ser. B},
  pages =         {52--75},
  title =         {Acyclic orientation polynomials and the sink theorem
                   for chromatic symmetric functions},
  volume =        {149},
  year =          {2021},
}

@article{KS25,
  author =        {Keevash, P. and Staden, K.},
  journal =       {J. Eur. Math. Soc. (JEMS)},
  number =        {5},
  pages =         {1769--1826},
  title =         {{Ringel}'s tree packing conjecture in quasirandom
                   graphs},
  volume =        {27},
  year =          {2025},
}

@misc{Li26X,
  author =        {Li, {\relax E.Y.H}},
  howpublished =  {arXiv:2503.09484v2},
  title =         {A quantitative way to {$e$}-positivity of trees},
  year =          {2026},
}

@article{LLYZ25,
  author =        {Li, {\relax E.Y.H} and Li, {\relax G.M.X} and
                   Yang, {\relax A.L.B} and Zhang, {\relax Z.X}},
  journal =       {S\'{e}m. Lothar. Combin.},
  pages =         {Art. \#154, 11 pp.},
  title =         {Strongly nice property and {Schur} positivity of
                   graphs},
  volume =        {93B},
  year =          {2025},
}

@article{LQYZ25,
  author =        {Li, {\relax G.M.X} and Qiu, Dun and
                   Yang, {\relax A.L.B} and Zhang, Zhong-Xue},
  journal =       {SIAM J. Discrete Math.},
  number =        {4},
  pages =         {2250--2267},
  title =         {Stanley's conjecture on the {Schur} positivity of
                   distributive lattices},
  volume =        {39},
  year =          {2025},
}

@article{LS19,
  author =        {Martin Loebl and Jean-S{\'{e}}bastien Sereni},
  journal =       {Annales de l'Institut Henri Poincar\'e D},
  number =        {3},
  pages =         {357--384},
  title =         {Isomorphism of weighted trees and {S}tanley's
                   isomorphism conjecture for caterpillars},
  volume =        {6},
  year =          {2019},
}

@book{Lot97B,
  author =        {Lothaire, M.},
  edition =       {2nd},
  publisher =     {Camb. Univ. Press, Cambridge},
  series =        {Cambridge Mathematical Library},
  title =         {Combinatorics on words},
  year =          {1997},
}

@article{MMW08,
  author =        {Martin, {\relax J.L} and Morin, Matthew and
                   Wagner, {\relax J.D}},
  journal =       {J. Combin. Theory Ser. A},
  number =        {2},
  pages =         {237--253},
  title =         {On distinguishing trees by their chromatic symmetric
                   functions},
  volume =        {115},
  year =          {2008},
}

@article{SW16,
  author =        {Shareshian, {\relax J} and Wachs, {\relax M.L}},
  journal =       {Adv. Math.},
  pages =         {497--551},
  title =         {Chromatic quasisymmetric functions},
  volume =        {295},
  year =          {2016},
}

@article{Sta95,
  author =        {Stanley, {\relax R.P}},
  journal =       {Adv. Math.},
  number =        {1},
  pages =         {166--194},
  title =         {A symmetric function generalization of the chromatic
                   polynomial of a graph},
  volume =        {111},
  year =          {1995},
}

@article{Sta98,
  author =        {Stanley, {\relax R.P}},
  journal =       {Discrete Math.},
  number =        {1-3},
  pages =         {267--286},
  title =         {Graph colorings and related symmetric functions:
                   ideas and applications: a description of results,
                   interesting applications, \& notable open problems},
  volume =        {193},
  year =          {1998},
}

@article{SS93,
  author =        {Stanley, {\relax R.P} and Stembridge, {\relax J.R}},
  journal =       {J. Combin. Theory Ser. A},
  number =        {2},
  pages =         {261--279},
  title =         {On immanants of {J}acobi-{T}rudi matrices and
                   permutations with restricted position},
  volume =        {62},
  year =          {1993},
}

@article{Tom26,
  author =        {Tom, {\relax F}},
  journal =       {Adv. in Appl. Math.},
  pages =         {103044},
  title =         {Graphs missing a connected partition},
  volume =        {175},
  year =          {2026},
}

@article{TV26-vglue,
  author =        {Tom, {\relax F} and Vailaya, {\relax A}},
  journal =       {Electron. J. Combin.},
  number =        {2},
  pages =         {\#P2.9},
  title =         {The chromatic symmetric function of graphs glued at a
                   single vertex},
  volume =        {33},
  year =          {2026},
}

@article{WW20,
  author =        {Wang, {\relax D.G.L} and Wang, {\relax M.M.Y}},
  journal =       {Discrete Appl. Math.},
  pages =         {621--630},
  title =         {A combinatorial formula for the {S}chur coefficients
                   of chromatic symmetric functions},
  volume =        {285},
  year =          {2020},
}

@article{WW23-DAM,
  author =        {Wang, {\relax D.G.L} and Wang, {\relax M.M.Y}},
  journal =       {Discrete Appl. Math.},
  pages =         {226--240},
  title =         {The $e$-positivity and {Schur} positivity of some
                   spiders and broom trees},
  volume =        {325},
  year =          {2023},
}

@article{WW23-JAC,
  author =        {Wang, {\relax D.G.L} and Wang, {\relax M.M.Y}},
  journal =       {J. Algebraic Combin.},
  number =        {2},
  pages =         {495--514},
  title =         {The $e$-positivity of two classes of cycle-chord
                   graphs},
  volume =        {57},
  year =          {2023},
}

@article{WZ26,
  author =        {Wang, {\relax D.G.L} and Zhang, {\relax K}},
  journal =       {Discrete Appl. Math.},
  pages =         {1--18},
  title =         {Exact thresholds for {Schur} positivity of the
                   lattices $\mathbf m\times\mathbf 2$ and $\mathbf
                   m\times\mathbf 3$},
  volume =        {395},
  year =          {2026},
}

@phdthesis{Wol97D,
  author =        {Wolfgang, III, {\relax H.L}},
  school =        {Massachusetts Institute of Technology},
  title =         {Two Interactions between Combinatorics and
                   Representation Theory: Monomial Immanants and
                   Hochschild Cohomology},
  year =          {1997},
}

@article{Zhe22,
  author =        {Zheng, {\relax K}},
  journal =       {J. Combin. Theory Ser. A},
  pages =         {105608},
  title =         {On the $e$-positivity of trees and spiders},
  volume =        {189},
  year =          {2022},
}

\end{document}